\documentclass[]{article}

\usepackage{amssymb,amsmath,amsfonts,amscd,amsthm}
\usepackage{graphicx}
\usepackage[]{algorithm2e}
\usepackage{subfig}
\usepackage{caption}
\usepackage{url}

\newcommand{\gra}{\G(D,d)}
\newcommand{\dl}{\mathrm{dist}_{\G}(L_1,L_2)}
\newcommand{\ti}{\theta_i}
\newcommand{\sumi}{\sum_{i=1}^d}
\newcommand{\setl}{\{L_1,L_2,...,L_m\}}
\newcommand{\rc}{({\mathrm R}, {\mathrm c})}
\newcommand{\lsh}{{\em locality sensitive hashing}}
\newcommand{\pu}{\Pb[h\in\HH|h(L_1)=h(L_2)]}

\newtheorem{theorem}{Theorem}[section]
\newtheorem{cor}[theorem]{Corollary}
\newtheorem{corollary}[theorem]{Corollary}
\newtheorem{lemma}[theorem]{Lemma}
\newtheorem{sublemma}[theorem]{Sublemma}
\newtheorem{addendum}[theorem]{Addendum}
\newtheorem{prop}[theorem]{Proposition}
\newtheorem{proposition}[theorem]{Proposition}
\newtheorem{assumptions}[theorem]{Assumptions}
\newtheorem{assumption}[theorem]{Assumption}
\newtheorem{conjecture}[theorem]{Conjecture}
\newtheorem{defn}[theorem]{Definition}
\newtheorem{definition}[theorem]{Definition}
\newtheorem{definitions}[theorem]{Definitions}
\newtheorem{notation}[theorem]{Notation}
\newtheorem{rem}[theorem]{Remark}
\newtheorem{claim}[theorem]{Claim}
\newtheorem{remark}[theorem]{Remark}
\newtheorem{remarks}[theorem]{Remarks}
\newtheorem{example}[theorem]{Example}
\newtheorem{examples}[theorem]{Examples}
\newtheorem{fact}[theorem]{Fact}
\newtheorem{problem}[theorem]{Problem}
\newtheorem{observation}[theorem]{Observation}
\newtheorem{condition}[theorem]{Condition}

\newcommand{\acknowledgments}{{\em Acknowledgments.} }

\newcommand{\R}{\mathbb{R}}
\newcommand{\Z}{\mathbb{Z}}
\newcommand{\Q}{\mathbb{Q}}
\newcommand{\C}{\mathbb{C}}
\newcommand{\G}{\mathrm{G}}
\newcommand{\N}{\mathbb{N}}
\newcommand{\Pb}{\mathbb{P}}
\newcommand{\E}{\mathbb{E}}
\newcommand{\Sb}{\mathbb{S}}
\newcommand{\K}{\mathbb{K}}
\newcommand{\x}{\mathbf{x}}
\newcommand{\y}{\mathbf{y}}
\newcommand{\malpha}{\boldsymbol\alpha}
\newcommand{\mtheta}{\boldsymbol\theta}

\newcommand{\B}{B}
\newcommand{\id}{\mathrm{id}}
\newcommand{\ind}{\mathrm{ind}}
\newcommand{\re}{\mathrm{re}}
\newcommand{\im}{\mathrm{im}}
\newcommand{\var}{\mathrm{Var}}
\newcommand{\conv}{\mathrm{conv}}
\newcommand{\rank}{\mathrm{rank}}
\newcommand{\End}{\mathrm{End}}
\newcommand{\tr}{\mathrm{tr}}
\newcommand{\aff}{\mathrm{aff}}

\newcommand{\BB}{\mathcal{B}}
\newcommand{\HH}{\mathcal{H}}
\newcommand{\EE}{\mathcal{E}}
\newcommand{\LL}{\mathcal{L}}
\newcommand{\OO}{\mathcal{O}}
\newcommand{\Sc}{\mathcal{S}}
\newcommand{\XX}{\mathcal{X}}
\newcommand{\DD}{\mathcal{D}}
\newcommand{\MM}{\mathcal{M}}
\newcommand{\JJ}{\mathcal{J}}
\newcommand{\GG}{\mathcal{G}}
\newcommand{\FF}{\mathcal{F}}
\newcommand{\PP}{\mathcal{P}}
\newcommand{\CC}{\mathcal{C}}
\newcommand{\NN}{\mathcal{N}}
\newcommand{\RR}{\mathcal{R}}

\newcommand{\Top}{\mathfrak T}

\newcommand{\bb}{\mathbb}
\newcommand{\mb}{\mathbf}
\newcommand{\mc}{\mathcal}
\newcommand{\mf}{\mathfrak}

\newcommand{\eps}{\varepsilon}
\newcommand{\<}{\langle}
\renewcommand{\>}{\rangle}
\newcommand{\tensor}{\otimes}
\newcommand{\Rgeq}{\R^{\scriptscriptstyle \geq 0}}
\newcommand{\Rleq}{\R^{\scriptscriptstyle \leq 0}}

\newcommand{\bernoulli}[1]{\mathrm{Ber}(#1)}
\newcommand{\normal}[2]{\mc N \left( #1 , #2 \right) }

\newcommand{\minimize}[2]{\text{minimize} \quad #1 \quad \text{subject to} \quad #2}
\newcommand{\prob}[1]{\P\left[ \, #1 \, \right] }
\newcommand{\condprob}[2]{\Pb_{\mu}\left[ \, #1 \mid #2 \, \right] }
\newcommand{\range}[1]{\mathrm{range}(#1)}
\renewcommand{\dim}[1]{\mathrm{dim}(#1)}
\newcommand{\norm}[2]{\left\| #1 \right\|_{#2}}
\newcommand{\expect}[1]{\E\left[ #1 \right]}
\newcommand{\condexp}[2]{\E\left[ #1 \mid #2 \right]}
\newcommand{\innerprod}[2]{\left\< #1, #2 \right\>}
\newcommand{\expfrac}[2]{\exp\left(-\frac{#1}{#2}\right)}
\newcommand{\magnitude}[1]{\left|#1\right|}
\newcommand{\set}[1]{\left\{ #1 \right\}}
\renewcommand{\vec}[2]{\mathrm{\text{#1}}\left[ #2 \right]}
\newcommand{\di}{{\,\mathrm{d}}}

\DeclareMathOperator{\vol}{Vol}
\DeclareMathOperator{\Span}{span}
\DeclareMathOperator{\dist}{dist}
\DeclareMathOperator{\Pmin}{Pmin}
\DeclareMathOperator{\Pmax}{Pmax}
\DeclareMathOperator{\ProjNorm}{ProjNorm}
\DeclareMathOperator{\Ang}{Ang}
\DeclareMathOperator{\argmin}{\mathrm{argmin}}

\begin{document}
\title{Nonparametric Bayesian Regression on Manifolds via Brownian Motion\footnote{This work was supported by NSF awards DMS-09-56072 and DMS-14-18386 and the University of Minnesota Doctoral Dissertation Fellowship Program.}}
\author{Xu Wang\\
       wang1591@umn.edu \\
       Dept.\ of Mathematics\\
       University of Minnesota\\
       Minneapolis, MN 55455
       \and
      Gilad Lerman\\
      lerman@umn.edu \\
      Dept.\ of Mathematics\\
      University of Minnesota\\
      Minneapolis, MN 55455\\
}
\maketitle

\begin{abstract}
This paper proposes a novel framework for manifold-valued regression and establishes its consistency as well as its contraction rate. It assumes a predictor with values in the interval $[0,1]$ and response with values in a compact Riemannian manifold $M$. This setting is useful for applications such as modeling dynamic scenes or shape deformations,  where the visual scene or the deformed objects can be modeled by a manifold. The proposed framework is nonparametric and uses the heat kernel (and its associated Brownian motion) on manifolds as an averaging procedure. It directly generalizes the use of the Gaussian kernel (as a natural model of additive noise) in vector-valued regression problems. In order to avoid explicit dependence on estimates of the heat kernel, we follow a Bayesian setting, where Brownian motion on $M$ induces a prior distribution on the space of continuous functions $C([0,1], M)$.
For the case of discretized Brownian motion, we establish the consistency of the posterior distribution in terms of the $L_{q}$ distances for any $1 \leq q < \infty$.
Most importantly, we establish contraction rate of order $O(n^{-1/4+\epsilon})$ for any fixed $\epsilon>0$, where $n$ is the number of observations. For the continuous Brownian motion we establish weak consistency.
\end{abstract}
\section{Introduction}

In many applications of regression analysis, the response variables lie in Riemannian manifolds.
For example, in directional statistics~\cite{mardia2009directional, fisher1993statistical, fisher1995statistical} the response variables take values in the sphere or the group of rotations. Applications of directional statistics include crystallography~\cite{Murshudov:ba5152}, altitude determination for navigation and guidance control~\cite{Shuster_Oh81}, testing procedure for Gene Ontology cellular component categories~\cite{Olsen_Science10}, visual invariance studies~\cite{Miao:2007:LLG:1288856.1288861} and geostatics~\cite{Webster08}.
Other modern applications of regression give rise to different types of manifold-valued responses. In the regression problem of estimating shape deformations of the brain over time (e.g., for studying brain development, aging or diseases), the response variables lie in the space of shapes~\cite{FishbaughPGD14, nielsen2009emerging, Hong12, Bhattacharya:2012:NIM:2331106, Nietha11,Gee13}.
In the analysis of landmarks~\cite{Hinkle14} the response variables lie in the Lie group of diffeomorphisms.


The quantitative analysis of regression with manifold-valued responses (which we refer to as manifold-valued regression) is still in early stages and is significantly less developed than
statistical analysis of vector-valued regression with manifold-valued predictors~\cite{aswani2011, CS02, NilssonSJ07, Roberto14, Pelle06, Yun13, Chen_Wu_2013}.
A main obstacle for advancing the analysis of manifold-valued regression is that there is no linear structure in general Riemannian manifolds
and thus no direct method for averaging responses.
Parametric methods for regression problems with manifold-valued responses~\cite{FishbaughPGD14, Hong12, Miao:2007:LLG:1288856.1288861, geode13, Hinkle14} directly generalize
the linear or polynomial real-valued regressions to geodesic or Riemannian polynomial manifold-valued regression.
Nevertheless, the geodesic or Riemannian polynomial assumption on the underlying function is often too restrictive and for many applications non-parametric models are required.
To address this issue, Hein~\cite{hein2009robust} and Bhattacharya~\cite{Bhattacharya:2012:NIM:2331106} proposed kernel-smoothing estimators, where in~\cite{hein2009robust} the
predictors and responses take values in manifolds and in~\cite{Bhattacharya:2012:NIM:2331106} the predictors and responses take values in compact metric spaces with special kernels.
Hein~\cite{hein2009robust} proved convergence of the risk function to a minimal risk (w.p.~1; conditioned on the predictor) and Bhattacharya~\cite{Bhattacharya:2012:NIM:2331106}
established consistency of the joint density function of the predictors and the responses.
However, the rate of contraction (that is, the rate at which the posterior distribution contracts to a $\delta$ distribution with respect to the underlying regression function)
of any previously proposed manifold-valued regression estimator was not established. To the best of our knowledge, rate of contraction was only established when both the predictor and response variables are real~\cite{MR2791382} and this work does not seem to extend to manifold-valued regression.


The main goal of this paper is to establish the rate of contraction of a natural estimator for manifold-valued regression (with real-valued predictors).
This estimator is proposed here for the first time.

\subsection{Setting for Regression with Manifold-Valued Responses}
\label{sec:setting}
We assume that the predictor $t$ takes values in $[0,1]$ and the response $x$ takes values in a compact $D$-dimensional Riemannian manifold $M$.
We denote the Riemannian measure on $M$ by $\mu$ ($d \mu$ is the volume form).
We also assume an underlying function
$f_0 \in C([0,1], M)$, which relates between the predictor variables and response variables by determining a density function $p_{f_0(t)}(x)$, so that
\begin{equation}
\label{equ:model0}
x | t \sim p_{f_0(t)}(x).
\end{equation}

We find it natural to define
\begin{equation}\label{equ:model}
p_{f_0(t)}(x)=p_{\sigma^2}(f_0(t), x),
\end{equation}
where $p_{\sigma^2}(f_0(t), x)$ denotes the heat kernel on $M$ centered at $f_0(t)$ and evaluated at time $\sigma^2$.
Equivalently, $p_{\sigma^2}(f_0(t), x)$ is the transition probability of Brownian motion on $M$ (with the measure $\mu$) from $f_0(t)$ to $x$ at time $\sigma^2$.
We note that $\sigma^2$ controls the variance of the distribution of $x | t$ and as $\sigma^2 \rightarrow 0$, the distribution of $x | t$ approaches $\delta_{f_0(t)}$.
In the special case where $M=\R^D$:
\[
p_{\sigma^2}(f_0(t),\mb{x})= \frac{1}{(\sqrt{2\pi} \sigma)^D }\exp \left(\frac{-\|\mb{x}-f_0(t)\|^2}{2\sigma^2}\right),
\]
and this implies the common model: $x-f_0(t)\,|\,t \sim N(\mb{0},\sigma^2 \mb{I})$.

We also assume a distribution $p(t)$ of $t$, whose support equals $[0,1]$, though its exact form is irrelevant in the analysis. At last, we assume $n$ i.i.d.~observations $\{(t_i, x_i)\}_{i=1}^n \subset [0,1] \times M$ with the joint distribution $P_0^n$ and the density function
\begin{equation}
\label{eq:p0N}
p_0^n = \prod_{i=1}^n p(t_i) p_{\sigma^2}(f_0(t_i), x_i).
\end{equation}
{\it The aim of the regression problem is to estimate $f_0$ among all functions in $C([0,1], M)$ given the observations $\{(t_i, x_i)\}_{i=1}^n$.}

For simplicity, we denote throughout the rest of the paper
$$\PP := C([0,1], M).$$

\subsection{Bayesian Perspective: Prior and Posterior Distributions Based on the Brownian Motion}\label{subsec:prior}
Since the set of functions $\PP$ includes Brownian paths, the heat kernel, which expresses the Brownian transition probability, can be used to form a prior distribution on $\PP$.
For the sake of clarity, we need to distinguish between two different ways of using the heat kernel in this paper.
The first one applies the heat kernel $p_{\sigma^2}(f_0(t), x)$ with $t \in [0,1]$, $f_0 \in \PP$ and $x \in M$ (see e.g., Section~\ref{sec:setting}), where the time (or variance) parameter $\sigma^2$ quantifies the ``noise'' in $x$ w.r.t.~the underlying function $f_0(t)$.
The second one uses the heat kernel $p_h(x,y)$ with $h \in \R_{+}$ and $x,y \in M$, where the time parameter $h$ inversely characterizes the ``smoothness'' of the path between $x$ and $y$. The smaller $h$, the smoother the path between $x$ and $y$ (since smaller $h$ makes it less probable for $y$ to get further away from $x$).
Using the heat kernel $p_h(x,y)$, we define in Section~\ref{sec:DefBMP} a continuous Brownian motion (BM) prior distribution and in Section~\ref{sec:DefDBMP} a discretized BM prior distribution.
Section~\ref{sec:posterior} then defines posterior distributions in terms of the prior distributions and the given observations $\{(t_i, x_i)\}_{i=1}^n \subset [0,1] \times M$ of the setting. 

\subsubsection{The Continuous BM Prior on $\PP$}
\label{sec:DefBMP}
We note that a function $f \in \PP$ can be identified as a parametrized path in $M$.
Let's assume that $x\in M$ is a starting point of this path, that is $f(0)=x$.
We denote $\PP_{x}:=\{f \in \PP : f(0)=x\}$. Corollary 2.19 of~\cite{BaerPfaeffle2012} implies that there exists a unique probability measure
$W_{x}$ on $\PP_{x}$ such that for any $n\in \N$, $0<t_1<...<t_n = 1$, and open subsets $U_1, \ldots, U_n \in M$, the following identify is satisfied
\begin{multline}
W_{x}(f \in \PP_{x} \ | \ f(t_1) \in U_1, \ldots, f(t_n) \in U_n) =\\
\int_{U_1 \times \ldots \times U_n}
 p_{t_n-t_{n-1}}(x_n,x_{n-1}) \cdots  p_{t_2-t_{1}} (x_2,x_{1}) p_{t_1}(x_1,x)
d \mu(x_1) \cdots d \mu(x_n).
\end{multline}
We define the conditional prior distribution of $f \in \PP$ given $x \in M$ by $W_x$.
We assume that the distribution of $f(0)=x$ is $\mu/ \mu(M)$ and thus obtain that the prior distribution $\Pi(f) $ of $f \in \PP$
is $W_x \times \mu/\mu(M)$.

\subsubsection{The Discretized BM Prior $\PP$}
\label{sec:DefDBMP}
The continuous BM prior often does not have a density function.
We discuss here a special case of discretized BM, where the density function of the prior is well-defined.
For $0<h<1$  such that $1/h$ is an integer, we define $PGF(h)$ as the set of {\it piecewise geodesic functions} from $[0,1]$ to $M$, where for each $0\leq k<1/h$, $k \in \N$, the interval $[kh, (k+1)h]$ is mapped to the geodesic curve from $f(kh)$ to $f((k+1)h)$. Each function in $PGF(h)$ is determined by its values at $f(kh)$. Let the distribution of $f(0)$ be uniform w.r.t. the Riemannian measure $\mu$
and let the transition probability from $f(kh)$ to $f((k+1)h)$ be given by the heat kernel $p_h(f(kh),f((k+1)h))$.
Then the density function $\pi_h$ (w.r.t. $\mu$) of the discretized BM prior on $PGF(h)$ can be specified as follows:
\[
\displaystyle \pi_h ( f ) = \frac{1}{\mu(M)}\prod_{k=1}^{1/h} p_h(f(kh-h), f(kh)).
\]
The corresponding distribution is denoted by $\Pi_h$.

Throughout the paper we assume a sequence $b_n\rightarrow 0$ with $0<b_n<1$ and with some abuse of notation denote by $\Pi_n$ the sequence of discretized BM priors defined above with $h=b_n$. By construction, $\Pi_n$ is supported on $PGF(b_n)$.
Since $PGF(b_n) \subset \PP$, $\Pi_n$ can also be considered as a set of priors on $\PP$.

\subsubsection{Posterior Distributions}
\label{sec:posterior}
Given observations $\{(t_i,x_i)\}_{i=1}^n$ drawn according to the setting of Section~\ref{sec:setting}, the posterior distribution of $\Pi$ has the density function
\begin{equation}\label{equ:posterior}
\begin{aligned}
\displaystyle \Pi(f\in A|\{(t_i,x_i)\}_{i=1}^n) &\propto \int_{f\in A}\prod_{i=1}^n p(t_i,x_i|f) d\Pi(f) \\
&=\int_{f\in A} \prod_{i=1}^n p_{f(t_i)}(x_i)p(t_i) d\Pi(f),
\end{aligned}
\end{equation}
where the equality in~\eqref{equ:posterior} follows by applying \eqref{equ:model0} and~\eqref{equ:model} to the estimator $f$ of $f_0$.



\subsection{Main Theorems: Posterior Consistency and Rate of Contraction}
\label{sec:theorems}

We establish the posterior consistency for the discretized and continuous BM priors respectively.
That is, we show that as $n$ approaches infinity, the posterior distributions contract with high probability to the distribution
$\delta_{f_0}$ (recall that $f_0$ is the underlying function in $\PP$).
Furthermore, for the discretized BM we study the rate of contraction of the posterior distribution.
The theorem for the discretized BM is formulated in Section~\ref{sec:theorems_discrete} and the one for
the continuous BM (with weaker convergence) in Section~\ref{sec:theorems_continuous}.

\subsubsection{Posterior Consistency and Rate of Contraction for Discretized BM}
\label{sec:theorems_discrete}

Theorem~\ref{theorem:rate} below formulates the rate of contraction of the posterior distribution of the discretized BM with respect to the $L_q$ metric on $\PP$, where $1 \leq q < \infty$. This metric, $d_q$, is defined as follows:
\begin{equation}\label{equ:distdefqp}
\displaystyle d_{q} (f_1, f_2) = \left(\int_{t\in [0,1]} \dist_M( f_1(t), f_2(t) )^q p(t) dt \right)^{1/q},
\end{equation}
where $\dist_M$ denotes the geodesic distance on $M$ and $p(t)$ is the pdf for the predictor $t$.

%
\begin{theorem}\label{theorem:rate} 
Assume a regression setting with a predictor variable $t \in [0,1]$, whose pdf $p(t)$ is strictly positive on $[0,1]$, a response variable $x$ in a compact finite-dimensional Riemannian manifold $M$ and
an underlying and unknown Lipschitz function $f_0 \in \PP$, which relates between $x$ and $t$ according to~\eqref{equ:model0} and~\eqref{equ:model}.
Assume an arbitrarily fixed $0<\epsilon<1/4$ and for $n \in \N$, let $b_n = n^{-1/2+2\epsilon}$ be the sidelength of the set $PGF(b_n)$
and let $\{\Pi_n\}_{n \in \N}$ denote the sequence of discretized BM priors on $PGF(b_n)$.
Then there exists an absolute constant $A_0$ and a fixed constant $C_0$ depending only on the positive minimum value of $p(t)$ on $[0,1]$, the volume of $M$ and
the Riemannian metric of $M$\footnote{More precisely, the dependence of the constant $C_{II}$ (which is later defined
in~\eqref{eq:lemma:boundp:5}) on the Riemannian metric.},
such that $\Pi_n(\cdot | \{(t_i, x_i)\}_{i=1}^n)$ contracts to $f_0$ according to the rate $\epsilon_n = \sqrt{b_n/C_0} = O(n^{-1/4+\epsilon})$.
More precisely, for any $1 \leq q < \infty$
\[
\Pi_n (f: d_{q} (f, f_0) \geq A_0 \epsilon_n | \{(t_i, x_i)\}_{i=1}^n) \rightarrow 0
\]
in $P_0^n$-probability (see~\eqref{eq:p0N}) as $n \rightarrow \infty$.
\end{theorem}

The proof of Theorem~\ref{theorem:rate} appears in Section~\ref{sec:proof:dbm} and utilizes a general strategy for establishing contraction according to~\cite{ghosal2000}.
The significance of the theorem is in properly determining the sidelength parameter $b_n$ (as a function of $n$).
Practical application of the discretized BM prior can suffer from underfitting or overfitting as a result of too small or too large choice of $b_n$ respectively. Theorem~\ref{theorem:rate} implies that for $n$ observations, $b_n$ should be picked as $n^{-1/2+2\epsilon}$ to achieve a contraction rate of $O(n^{-1/4+\epsilon})$ for any fixed $\epsilon>0$.


\subsubsection{Posterior Consistency for Continuous BM}
\label{sec:theorems_continuous}

We show here that the posterior distribution $\Pi(\pmb{\cdot} | \{(t_i, x_i)\}_{i=1}^n )$ is weakly consistent.
In order to clearly specify the weak convergence, it is natural to identify functions in $\PP$ with density functions of observations.
Let $\DD$ denote the set of densities $p(t,x)$ from which the observations $\{(t_i, x_i)\}_{i=1}^n \subset [0,1] \times M$ are drawn. Assuming a fixed variance $\sigma^2$, a function $f\in \PP$ can be identified with a density function $p_f\in \DD$ as follows:
\begin{equation}\label{map:contden}
\displaystyle \Phi: \quad f \longrightarrow p_f (t,x) := p_{\sigma^2}(f(t), x)p(t).
\end{equation}
Therefore, $\Pi$ induces a prior on the set $\DD$, which is again denoted by $\Pi$  with some abuse of notation. For the simplicity of analysis, we assume here that $\sigma^2$ is known. Section~\ref{sec:unknown_sigma} discusses the modification needed when $\sigma^2$ is unknown.

For the underlying function $f_0$, we define its weak neighborhood of radius $\epsilon$ by
\[
\displaystyle N_{\epsilon} (f_0) = \left\{ f \in \PP: \left| \int_{[0,1]\times M} p_g p_f dtd\mu(x) -  \int_{[0,1]\times M} p_g p_{f_0} dtd\mu(x) \right|\leq \epsilon,\, \forall g\in  \PP \right\}.
\]
Theorem~\ref{theorem:wcbm} states the weak posterior consistency of the continuous BM prior $\Pi$. It is proved later in Section~\ref{sec:proof:cbm}.
\begin{theorem}\label{theorem:wcbm}
If $M$ is a compact Riemannian manifold and if the true underlying function $f_0\in \PP$ of the regression model is Lipschitz continuous,
then the posterior distribution $\Pi(\pmb{\cdot} | \{(t_i, x_i)\}_{i=1}^n )$ is weakly consistent. In other words, for any $\epsilon>0$,
\[
\Pi( N_{\epsilon}(f_0) | \{(t_i, x_i)\}_{i=1}^n ) \longrightarrow 1
\]
almost surely w.r.t. the true probability measure $P_0^n$ (defined in~\eqref{eq:p0N}) as $n \rightarrow \infty$.
\end{theorem}

\subsection{Main Contributions of This Work}

The first contribution of this paper is the proposal of a natural model for manifold-valued regression (with real-valued predictors). 
Indeed, the heat kernel on the Riemannian manifold gives rise to an averaging process, which generalizes basic averages of vector-valued regression.
In particular, the heat kernel on $\R^D$ is the same as the Gaussian kernel (applied to the difference of $f(t)$ and $x$), which is widely used in regression when $x \in \R^D$ (due to
an additive Gaussian noise model).
The Bayesian setting is natural for the proposed model, since it uses the discretized
or continuous Brownian motion on $M$ as a prior distribution of $f$ and it does not directly use the heat kernel.
It is not hard to simulate the Brownian motion, but tight estimates of the heat kernel for general $M$ are hard.

The second and main contribution of this work is the derivation of the contraction rate of the posterior distribution for the discretized Brownian motion. To the best of our knowledge the rate of contraction was only established before for regression with real-valued predictors and responses. For this case, van Zanten~\cite{MR2791382} established contraction rate $n^{-1/4}$ for the posterior distribution of $n$ samples under the $L_p$-norm, where $1\leq p<\infty$. His analysis does not seem to extend to our setting. It is unclear to us if this stronger contraction rate also applies to the general case of manifold-valued regression (see discussion in Section~\ref{sec:better_contraction}).

The third contribution is the consistency result for the continuous Brownian motion.
The only other consistency result for manifold-valued regression we are aware of is by Bhattacharya~\cite{Bhattacharya:2012:NIM:2331106}. It suggests a general nonparametric Bayesian kernel-based framework for modeling the conditional distribution $x|t$, where the predictor $t$ and response $x$ take values in metric spaces with kernels. Under a suitable assumption on the kernels, \cite{Bhattacharya:2012:NIM:2331106} established the posterior consistency for the conditional distribution w.r.t.~the $L_1$ norm (see~\cite[Proposition~13.1]{Bhattacharya:2012:NIM:2331106}). We remark that \cite{Bhattacharya:2012:NIM:2331106} applies to responses and predictors in Riemannian manifolds (where the corresponding metric kernels are the heat kernels).
However, both the conditional distribution (of $x$ given $t$) and the prior distribution are different than the ones proposed here.
It is unclear how to obtain a rate of contraction for~\cite{Bhattacharya:2012:NIM:2331106}.

The last contribution is the implication of a new numerical procedure for manifold-valued regression, which is based on simulating a Brownian motion on $M$.
The flexibility of the shapes of the sample paths of the Brownian motion is advantageous over state-of-the-art geodesic regression methods. Real applications often do not give rise to geodesics and thus the nonparametric regression method is less likely to suffer from underfitting. Another nonparametric approach is kernel regression~\cite{hein2009robust, Bhattacharya:2012:NIM:2331106}. In Section~\ref{sec:exp}, we compare between kernel regression and Brownian motion regression (our method) for a particular example, which is easy to visualize.

\subsection{Organization of the Rest of the Paper}
The paper is organized as follows. Theorems~\ref{theorem:rate} and~\ref{theorem:wcbm} are proved in Sections~\ref{sec:proof:dbm} and~\ref{sec:proof:cbm} respectively. Section~\ref{sec:extension} extends the framework to the cases where $\sigma^2$ is unknown and $p(t)$ is supported on a subset of $[0,1]$. Section~\ref{sec:exp} demonstrates the performance of the proposed procedure on a particular example, which is easy to visualize, and
compares it to kernel regression~\cite{hein2009robust, Bhattacharya:2012:NIM:2331106}.

\section{Proof of Theorem~\ref{theorem:rate}}\label{sec:proof:dbm}

Our proof utilizes Theorem~2.1 of \cite[page 4]{ghosal2000}. The latter theorem establishes the contraction rate for a sequence of priors $\Pi_n$ over the set $\DD$ of joint densities of the predictor $t$ and response $x$ under some conditions on $\Pi_n$ and the covering number of $\DD$. We thus conclude Theorem~\ref{theorem:rate} by establishing these conditions.

We use the following distance $d_{q, \DD}$ on the space $\DD$ with an arbitrarily fixed $1 \leq q < \infty$:
\[
\displaystyle d_{q, \DD}(p_1, p_2) = \frac{1}{2}\|p_1-p_2\|_q \quad \text{for } p_1,p_2\in \DD.
\]
The regression framework is formulated in terms of the space $\PP$ (see Section~\ref{sec:setting}, in particular, the mapping of $\PP$ to $\DD$ in~\eqref{map:contden}) and the metric $d_{q}$ on $\PP$ (see~\eqref{equ:distdefqp}).
We also use the $d_{\infty}$ metric on $\PP$, which is defined by
\begin{equation}\label{equ:distdef}
\displaystyle d_{\infty} (f_1, f_2) = \max_{t \in [0,1] } \dist_M (f_1(t), f_2(t)),
\end{equation}

The proof is organized as follows. Section~\ref{sec:boundp} shows that under the mapping~\eqref{map:contden} of $\PP$ to $\DD$, $d_{q, \DD}$ is bounded from below by $d_{q}$
(and above by $d_{\infty}$). Therefore, the posterior contraction w.r.t.~$d_{q, \DD}$ implies the posterior contraction w.r.t.~$d_{q}$. Then, Sections~\ref{sec:density1}-\ref{sec:density3} show that if the sidelengths $\{b_n\}_{n \in \N}$ and a constant $\alpha>0$ are chosen properly, then the priors $\{\Pi_n\}_{n \in \N}$ and the sieve of functions $\{\PP_{n,\alpha}\}_{n \in \N}$ (defined later in~\eqref{eq:sieve}) satisfy conditions (2.2)-(2.4) respectively in Theorem~2.1 of~\cite{ghosal2000}. The posterior contraction of $\Pi_n$ is then concluded.

\subsection{Relations between $d_{q, \DD}$, $d_{q}$ and $d_{\infty}$}
\label{sec:boundp}
We formulate and prove the following lemma, which relates between $d_{q, \DD}$,  $d_{q}$ and $d_{\infty}$.
It is later used as follows: The first inequality of~\eqref{eq:lemma:boundp} deduces $L_q$ convergence in $\PP$ from $L_q$ convergence in $\DD$. The second inequality of~\eqref{eq:lemma:boundp} is used in finding the covering number of the space $\DD$.

\begin{lemma}\label{lemma:boundp}
If $0<m_p$, $M_p \in \R$ and $m_p \leq p(t) \leq M_p$ for all $t\in [0,1]$, then there exists two constants $C_0, C_1>0$ depending only on $m_p$, $M_p$ and the Riemannian manifold $M$
such that for any $f_1, f_2 \in \PP$ with corresponding densities $p_{f_1}$, $p_{f_2}$ in $\DD$ (via~\eqref{map:contden})
\begin{equation}\label{eq:lemma:boundp}
C_0 d_{q} (f_1, f_2) \leq d_{q, \DD} (p_{f_1}, p_{f_2}) \leq C_1 d_{\infty} (f_1, f_2).
\end{equation}
\end{lemma}

\begin{proof}
For $x_1\not= x_2$, we define the function
\begin{equation}
\label{eq:def_F}
\displaystyle F(x_1, x_2,y)=\frac{\left| p_{\sigma^2} (x_1, y)-p_{\sigma^2}(x_2,y)\right|}{\dist_M (x_1,x_2)}.
\end{equation}
We note that the first inequality of~\eqref{eq:lemma:boundp} is true if there exists a constant $C_0>0$ such that
\begin{equation}\label{eq:lemma:boundp:1}
\displaystyle \int_{y\in M} F(x_1,x_2,y)^q d\mu(y) \geq \frac{C_0^q}{m^{q-1}_p},\quad \forall x_1\not=x_2\in M.
\end{equation}

Since $M$ is compact and $p_{\sigma^2}(x,y)$ is infinitely differentiable, for any $\epsilon>0$, there exists $\delta>0$ such that
\begin{equation}\label{eq:lemma:boundp:2}
\displaystyle \left| F(x_1,x_2,y) - \left|\frac{\partial p_{\sigma^2}(x_1, y)}{\partial v_{12}}\right| \right| \leq \epsilon, \quad \forall x_1,x_2,y\in M,\,\dist_M(x_1, x_2)\leq \delta,
\end{equation}
where $v_{12}\in T_{x_1} M$ is the unit vector of the geodesic connecting $x_1$ and $x_2$. Since the heat kernel $p_{\sigma^2}(x_1, y)$ is not constant and due to the compactness of the space of unit tangent vectors, there exists $C'_0>0$ such that
\begin{equation}\label{eq:lemma:boundp:3}
\displaystyle \int_{y\in M}\left|\frac{\partial p_{\sigma^2}(x_1, y)}{\partial v_{12}}\right| d\mu(y) \geq C'_0, \quad \forall x_1\in M, v_{12}\in T_{x_1} M, \|v_{12}\|=1.
\end{equation}
Inequalities~\eqref{eq:lemma:boundp:2}, \eqref{eq:lemma:boundp:3} and the Schwarz inequality imply that
\begin{equation}\label{eq:lemma:boundp:4}
\displaystyle \int_{y\in M} F(x_1,x_2,y)^q d\mu(y) \geq C_I:=\left(\frac{C'_0-\epsilon \mu(M)}{\mu(M)}\right)^q, \,\, \forall x_1,x_2,y\in M,\, \dist_M(x_1,x_2)\leq \delta.
\end{equation}
If we pick $\epsilon$ small enough (with its $\delta$ in \eqref{eq:lemma:boundp:2}), $C_I$ is a positive number. On the other hand, if the pair $(x_1,x_2)$ satisfies that $\dist_M(x_1,x_2)\geq \delta$, we show that for some constant $C_{II}>0$,
\begin{equation}\label{eq:lemma:boundp:5}
\displaystyle \int_{y\in M} F(x_1,x_2,y)^q d\mu(y) \geq C_{II}, \,\, \forall x_1,x_2,y\in M,\, \dist_M(x_1,x_2)\geq \delta.
\end{equation}
Since the set $\{(x_1,x_2) | \dist_M (x_1,x_2)\geq \delta\}$ is compact, the existence of $C_{II}$ is guaranteed if we can show that
$$
\displaystyle \int_{y\in M} F(x_1,x_2,y)^q d\mu(y) >0, \,\, \forall x_1,x_2,y\in M,\, \dist_M(x_1,x_2)\geq \delta,
$$
which can further be reduced to showing that given any pair $(x_1,x_2) \in M^2$,
\begin{equation}\label{eq:lemma:boundp:6}
\exists y\in M, \quad p_{\sigma^2}(x_1,y) \not= p_{\sigma^2}(x_2,y).
\end{equation}
We prove~\eqref{eq:lemma:boundp:6} by contradiction. If~\eqref{eq:lemma:boundp:6} is not true, then
\begin{equation}\label{eq:lemma:boundp:7}
p_{\sigma^2}(x_1,y) = p_{\sigma^2}(x_2,y), \quad \forall y\in M.
\end{equation}
If we plug $y=x_1$ and $y=x_2$ respectively in~\eqref{eq:lemma:boundp:7}, and use the symmetry of the heat kernel, to get $p_{\sigma^2}(x_1, x_1) = p_{\sigma^2}(x_1, x_2) = p_{\sigma^2}(x_2, x_2)$, which means that
\begin{equation}\label{eq:lemma:boundp:8}
p_{\sigma^2}(x_1, x_2) = \sqrt{p_{\sigma^2}(x_1, x_1) p_{\sigma^2}(x_2, x_2)}.
\end{equation}
On the other hand,
\begin{equation}\label{eq:lemma:boundp:9}
\begin{aligned}
p_{\sigma^2}(x_1, x_2) &= \int_{z\in M} p_{\sigma^2/2}(x_1, z)p_{\sigma^2/2}(z, x_2) d\mu(z) \\
 &\leq  \sqrt{\int_{z\in M} p_{\sigma^2/2}(x_1, z)^2 d\mu(z) \int_{z\in M} p_{\sigma^2/2}(z, x_2)^2 d\mu(z)} \\
& = \sqrt{p_{\sigma^2}(x_1, x_1) p_{\sigma^2}(x_2, x_2)}.
\end{aligned}
\end{equation}
In view of~\eqref{eq:lemma:boundp:8} the Cauchy-Schwartz inequality used in~\eqref{eq:lemma:boundp:9} is an equality and consequently
$$
p_{\sigma^2/2}(x_1, z) = p_{\sigma^2/2}(x_2, z), \quad \forall z\in M.
$$
Applying the same argument iteratively, we conclude that for any $m>0$,
$$
p_{\sigma^2/2^m}(x_1, z) = p_{\sigma^2/2^m}(x_2, z), \quad \forall z\in M.
$$
However, as $m\rightarrow \infty$, $p_{\sigma^2/2^m}(x_1, z) \rightarrow \delta_{x_1}$ but $p_{\sigma^2/2^m}(x_2, z) \rightarrow \delta_{x_2}\not=\delta_{x_1}$. This is a contradiction. Inequality~\eqref{eq:lemma:boundp:6} and thus~\eqref{eq:lemma:boundp:5} are proved.
We conclude from~\eqref{eq:lemma:boundp:4} and~\eqref{eq:lemma:boundp:5}, the first inequality of~\eqref{eq:lemma:boundp} with $C_0=\left(\min(C_I, C_{II})m^{q-1}_p\right)^{1/q}$.

Next, we establish the second inequality of~\eqref{eq:lemma:boundp}. Theorem~4.1.4 in~\cite[page 105]{sam} states that $p_{\sigma^2}(x, y)$ is infinitely differentiable in both variables $x$ and $y$. In particular, its first partial derivatives are continuous. Furthermore, the fact that $M$ is compact implies that the first partial derivatives are bounded. That is, there exists $C_M >0$ such that
\[
\displaystyle \left|\frac{\partial p_{\sigma^2}(x, y)}{\partial x}\right| \leq C_M, \quad \left|\frac{\partial p_{\sigma^2}(x, y)}{\partial y}\right| \leq C_M.
\]
Consequently,
\begin{equation}
\label{eq:good_recovery}
\displaystyle | p_{\sigma^2}(x_1, y) - p_{\sigma^2}(x_2, y)| \leq C_M \dist_M (x_1, x_2).
\end{equation}
Applying~\eqref{eq:good_recovery} and then bounding $p(t)$ by $M_p$ and $\dist_M$ by $d_{\infty}$, we conclude~\eqref{eq:lemma:boundp:2} with $C_1= C_M M_p^{(q-1)/q} \mu(M)$ as follows:
\begin{equation}
\begin{aligned}
\nonumber
\displaystyle d_{q,\DD}(p_{f_1}, p_{f_2}) &= \left(\iint \left| p_{\sigma^2}(f_1(t), y)p(t) - p_{\sigma^2}(f_2(t), y)p(t) \right|^q  d\mu(y)dt\right)^{1/q} \\
\nonumber
 &\leq \left(\iint C_M^q \dist_M (f_1(t), f_2(t))^q p(t)^q d\mu(y)dt\right)^{1/q} \\
 &\leq  C_M M_p^{(q-1)/q} \mu(M) d_{\infty}(f_1, f_2).
\end{aligned}
\end{equation}
\end{proof}

\begin{remark}
We note that when $q=1$, the constants $C_0, C_1$ in Lemma~\ref{lemma:boundp} are independent of $p(t)$. In particular, in this case the condition $m_p \leq p(t) \leq M_p$ is not needed.
\end{remark}

\subsection{Verification of Inequality~2.2 of~\cite{ghosal2000}}\label{sec:density1}
We estimate the covering numbers of special subsets of $\PP$ and $\DD$. The final estimate verifies inequality~2.2 of~\cite{ghosal2000}. We start with some notation and definitions that also include these special subsets of $\PP$ and $\DD$.

For $0 < \alpha \leq 1$ and $f \in \PP$, let
\[
\displaystyle \|f\|_{\alpha}:=\max_{t_1 ,t_2 \in [0,1]} \frac{\dist_M(f(t_1),f(t_2))}{|t_1-t_2|^{\alpha}}
\]
and
$$\PP_{\alpha}:=\{f\in \PP | \, \|f\|_{\alpha} <\infty\}.$$
For a sequence $\{M_n\}_{n \in \N}$ increasing to infinity
we define the sieve of functions
\begin{equation}\label{eq:sieve}
\PP_{n,\alpha}=\{f\in \PP_{\alpha} | \, \|f\|_{\alpha}\leq M_n \}.
\end{equation}
This induces a sieve of densities $\DD_{n,\alpha}$ of $\DD_{\alpha}$ by the map~\eqref{map:contden}.
For $\epsilon>0$ and a metric space $\EE$ with the metric $d$, we denote by
$N(\epsilon, \EE, d)$ the $\epsilon$-covering number of $\EE$, which is the minimal number of balls of radius $\epsilon$ needed to cover $\EE$.

In the rest of the section we estimate the covering numbers of the sets $M$, $\PP_{n,\alpha}$ and
$\DD_{n,\alpha}$.
We assume a decreasing sequence $\epsilon_n$ approaching zero.
Section~\ref{sec:eps_net_M} upper bounds $N(\epsilon_n, M, \dist_M)$ for an arbitrary such sequence $\epsilon_n$. Section~\ref{sec:eps_net_P} upper bounds $N(\epsilon_n, \PP_{n,\alpha}, d_{\infty})$ for arbitrary sequences $\epsilon_n$ and $M_n$ as above.
At last, Section~\ref{sec:eps_net_D} upper bounds $N(\epsilon_n, \DD_{n,\alpha}, d_{q, \DD})$
for sequences $\epsilon_n$ and $M_n$ satisfying an additional condition (see~\eqref{equ:condition1} below). It verifies inequality~2.2 of~\cite{ghosal2000}.

\subsubsection{Covering Numbers of $M$}
\label{sec:eps_net_M}
For any $\epsilon_n>0$, we construct an $\epsilon_n$-net on the $D$-dimensional compact Riemannian manifold $M$. Let $\text{D}(M)$ be the diameter of $M$. That is,
\[
\displaystyle \text{D}(M)=\max_{x,y\in M} \dist_M(x,y).
\]
The Nash embedding theorem~\cite{Nash54} and Whitney embedding theorem~\cite{Whitney44} imply that there exists an isometric map
\[
E : M \longrightarrow \R^{2D}.
\]
Since $\text{D}(E(M)) \leq \text{D}(M)$, the image $E(M)$ is contained in an hypercube $HC$ with side length $2\text{D}(M)$. We partition this $HC$ as a regular grid with grid spacing $\epsilon_n/\sqrt{2D}$ in each direction. Since each point in $HC$ has distance less than $\epsilon_n$ to some grid vertex, the set of grid vertices, $GV(\epsilon_n)$, is an $\epsilon_n$-net of $HC$. Thus the $\epsilon_n$-covering number of $HC$ can be bounded as follows:
\begin{equation}\label{equ:cover}
\displaystyle N(\epsilon_n, HC, \dist_{\R^{2D}}) \leq \left( \frac{2\text{D}(M)}{\epsilon_n/\sqrt{2D}} \right)^{2D}.
\end{equation}

Next, we construct an $\epsilon_n$-net of $M$ using the $\epsilon_n/3$-net $GV(\epsilon_n/3)$ of $HC$. To begin with, we show in Lemma~\ref{lemma:distance} that the Riemannian distance and the Euclidean distance are equivalent locally under an isometric embedding.

\begin{lemma}\label{lemma:distance}
Let $M$ be a compact Riemannian manifold and $E$ be an isometric embedding to $\R^{2D}$. Then for any fixed constant $C > 0$, there exists a constant $\delta_C > 0$ such that $\forall x,y\in M$ with $\dist_{\R^{2D}} (E(x),E(y)) < \delta_C$,
\[
| \dist_M (x,y) - \dist_{\R^{2D}} (E(x), E(y)) | < C \dist_{\R^{2D}} (E(x), E(y)).
\]
\end{lemma}

\begin{proof}
Suppose this is not true. Then there exists a sequence of $(x_n, y_n) \in M^2$ such that $\dist_{\R^{2D}} (E(x_n), E(y_n)) \rightarrow 0$ and
\begin{equation}\label{equ:assumption}
| \dist_M (x_n,y_n) - \dist_{\R^{2D}} (E(x_n), E(y_n)) | \geq C \dist_{\R^{2D}} (E(x_n), E(y_n)).
\end{equation}
Since $M$ is compact, there is a subsequence, denoted again by $(x_n, y_n)$, and a point $z\in M$ such that $x_n,y_n \rightarrow z$. By picking an orthonormal basis of the tangent space $T_z M$ and using the exponential map $\exp_z$, one has normal coordinates
\[
\Phi: \R^D \equiv T_z M \supset B_z(\mb{0}, r) \longrightarrow M
\]
where $B_z(\mb{0}, r)$ is the $r$-ball centered the origin on $T_z M$. Let $\log_z = \exp_z^{-1}$ be the logarithm map at $z$ and $\dist_I$ be the Euclidean distance on $T_z M$. Let $\mb{x}_n=\log_z(x_n)$ and $\mb{y}_n=\log_z(y_n)$. Applying Lemma~12 in~\cite[page 24]{Xu14} for $\mb{x}_n, \mb{y}_n$,
\begin{equation}\label{equ:local1}
| \dist_M (x_n, y_n) - \dist_I (\mb{x}_n, \mb{y}_n) | < O(\max\{\|\mb{x}_n\|_2^2,\|\mb{y}_n\|_2^2\}) \dist_I (\mb{x}_n, \mb{y}_n).
\end{equation}
Let $f$ be the composition of $\Phi$ with $E$,
\[
f : \quad \R^D \supset B_z(\mb{0}, r) \longrightarrow \R^{2D}.
\]
We note that $f(\mb{x}_n) = E(x_n)$ and $f(\mb{y}_n)= E(y_n)$. The Tyler series of $f$ is
\[
\displaystyle f(\mb{y}_n)-f(\mb{x}_n) = (\nabla f(\mb{x}_n))^T (\mb{y}_n - \mb{x}_n)+ \frac{1}{2} (\mb{y}_n - \mb{x}_n)^T (\nabla^2 f(\mb{x}_n)) (\mb{y}_n - \mb{x}_n)+\cdots.
\]
This implies that
\begin{equation}\label{eq:lemma:distance1}
\| f(\mb{y}_n)-f(\mb{x}_n) \|_2 = \| (\nabla f(\mb{x}_n))^T (\mb{y}_n - \mb{x}_n) \|_2 + O(\| \mb{y}_n - \mb{x}_n \|_2^2).
\end{equation}
On the one hand, since $E$ is an isometric embedding, the linear map
\[
\nabla f(\mb{0}) : \quad \R^D \longrightarrow \R^{2D}
\]
preserves the Euclidean distance. On the other hand, the smoothness of $f$ implies that $\nabla f(\mb{x})$ has bounded derivatives. Thus,
\begin{equation}\label{eq:lemma:distance2}
\nabla f(\mb{x}_n) = \nabla f(\mb{0})+O(\|\mb{x}_n\|_2).
\end{equation}
Then, \eqref{eq:lemma:distance1} and~\eqref{eq:lemma:distance2} and the triangle inequality imply that
\[
\| f(\mb{y}_n)-f(\mb{x}_n) \|_2 = \| \mb{y}_n - \mb{x}_n \|_2 + O(\|\mb{x}_n\|_2\| \mb{y}_n - \mb{x}_n \|_2+\| \mb{y}_n - \mb{x}_n \|_2^2).
\]
In other words,
\begin{equation}\label{equ:local2}
| \dist_I (\mb{x}_n, \mb{y}_n) - \dist_{\R^{2D}} (E(x_n), E(y_n)) | \leq O(\|\mb{x}_n\|_2+\| \mb{y}_n - \mb{x}_n \|_2) \dist_I (\mb{x}_n, \mb{y}_n).
\end{equation}
By \eqref{equ:local1} and \eqref{equ:local2},
\[
| \dist_M (x_n, y_n) - \dist_{\R^{2D}} (E(x_n), E(y_n)) | < c_n \dist_I (\mb{x}_n, \mb{y}_n),
\]
where $c_n= O(\|\mb{x}_n\|_2+\| \mb{y}_n - \mb{x}_n \|_2+\max\{\|\mb{x}_n\|_2^2,\|\mb{y}_n\|_2^2\})$. Moreover, by \eqref{equ:local2},
\[
\dist_I (\mb{x}_n, \mb{y}_n) < (1-O(\|\mb{x}_n\|_2+\| \mb{y}_n - \mb{x}_n \|_2))^{-1} \dist_{\R^{2D}} (E(x_n), E(y_n)).
\]
Therefore, if $c'_n = c_n/(1-O(\|\mb{x}_n\|_2+\| \mb{y}_n - \mb{x}_n \|_2))$, then
\[
| \dist_M (x_n, y_n) - \dist_{\R^{2D}} (E(x_n), E(y_n)) | < c'_n \dist_{\R^{2D}} (E(x_n), E(y_n)).
\]
We note that $c'_n \rightarrow 0$ as $n\rightarrow \infty$ since $\mb{x}_n, \mb{y}_n\rightarrow \mb{0}$ and this contradicts assumption~\eqref{equ:assumption}.

\end{proof}

Now, we construct an $\epsilon_n$-net of $M$ from $GV(\epsilon_n/3)$.
\begin{lemma}\label{lemma:coverbound}
Let $\displaystyle \widehat{GV}(\epsilon_n/3)=\{ x\in GV(\epsilon_n/3) | \dist_{\R^{2D}}(x, M) \leq \epsilon_n/3 \}.$
There exists a constant $\delta>0$ such that $E^{-1}(\mathrm{Proj}_{E(M)}(\widehat{GV}(\epsilon_n/3)))$ is an $\epsilon_n$-net of $M$ when $\epsilon_n < \delta$. Consequently,
\[
N(\epsilon_n, M, \dist_M) \leq N(\epsilon_n/3, HC, \dist_{\R^{2D}}).
\]
\end{lemma}

\begin{proof}
Suppose $\epsilon_n < \delta := \delta_{1/3}$ where $\delta_{1/3}$ is the constant $\delta_C$ in Lemma~\ref{lemma:distance} with $C=1/3$. For any point $x\in M$, let $y$ be the vertex in $GV(\epsilon_n/3)$ that is closest to $E(x)$ w.r.t. $\dist_{\R^{2D}}$. Then, by definition,
$y\in \widehat{GV}(\epsilon_n/3)$. Let $z = E^{-1} (\mathrm{Proj}_{E(M)} (y))$. To prove the lemma, it is sufficient to show that
\begin{equation}
\label{eq:about_to_end}
\dist_M (x, z) < \epsilon_n.
\end{equation}

Since $\dist_{\R^{2D}} (E(x), E(z)) < 2\epsilon_n/3<\delta_{1/3}$, Lemma~\ref{lemma:distance} states that
\begin{equation}\label{eq:lemma:coverbound1}
\displaystyle | \dist_M (x,z) - \dist_{\R^{2D}} (E(x), E(z)) | < \frac{1}{3} \dist_{\R^{2D}} (E(x), E(z)).
\end{equation}
Inequality~\eqref{eq:lemma:coverbound1} implies~\eqref{eq:about_to_end} and thus the lemma as follows:
\[
\displaystyle \dist_M (x,z) < \frac{4}{3}\dist_{\R^{2D}} (E(x), E(z))<8\epsilon_n/9.
\]

\end{proof}

From now on, we fix an $\epsilon_n/3$-net $S_n$ of $M$, generated as above from the projection of regular grid vertices $GV(\epsilon_n/9)$ of $HC$ with grid spacing $\epsilon_n/(9\sqrt{2D})$. Lemma~\ref{lemma:neighbor} provides an upper bound of the number of points in $S_n$ in the $\epsilon_n$-neighborhood of $x\in S_n$.

\begin{lemma}\label{lemma:neighbor}
For $x\in S_n$ and $X:= \{ y\in S_n | \dist_M (x, y) \leq \epsilon_n\}$,
\[
\# X \leq 21^{2D}.
\]
\end{lemma}

\begin{proof}
If $y\in X\subset S_n$, then there is a point $z \in GV(\epsilon_n/9)$ such that
\begin{equation}\label{eq:lemma:neighbor1}
E^{-1}(\mathrm{Proj}_{E(M)}(z))=y \quad \text{and}\quad \dist_{\R^{2D}} (z,E(y)) \leq \epsilon_n/9.
\end{equation}
We note that
\begin{equation}\label{eq:lemma:neighbor2}
\dist_{\R^{2D}} (E(x), E(y))\leq \dist_{M} (x, y)
\end{equation}
since $E$ is an isometric embedding. Inequalities~\eqref{eq:lemma:neighbor1} and~\eqref{eq:lemma:neighbor2} and the triangle inequality imply that $\dist_{\R^{2D}} (E(x), z) \leq 10\epsilon_n/9$. Thus, if
\[
Y:= \{z\in GV(\epsilon_n/9) | \dist_{\R^{2D}} (E(x), z) \leq 10\epsilon_n/9\},
\]
then $X \subset E^{-1}(\mathrm{Proj}_{E(M)} ( Y) )$. Since the grid spacing is $\epsilon_n/9$, $\# X \leq \# Y = 21^{2D}$.

\end{proof}

\subsubsection{Covering Numbers of $\PP_{n,\alpha}$}
\label{sec:eps_net_P}

Recall that $PGF(a)\subset \PP_{\alpha}$ is the set of piecewise geodesic functions which map each interval $[ka, (k+1)a]$ to a geodesic on M for $0 \leq k < 1/a$. We define
\[
F_{S_n}(a) = \{ f\in PGF(a) | f(ka)\in S_n\, \text{for } 0\leq k<1/a\},
\]
where $S_n$ was defined just before Lemma~\ref{lemma:neighbor}. The following Lemma upper bounds $N(\epsilon_n, \PP_{n,\alpha}, d_{\infty})$. It uses the constant $\delta_{1/9}$ which was defined in Lemma~\ref{lemma:distance} (here $C=1/9$).
\begin{lemma}\label{lemma:function1}
If $M$ is a $D$-dimensional compact Riemannian manifold with diameter $\text{D}(M)$, two sequences $\{M_n\}_{n\in \N}, \{\epsilon_n\}_{n\in \N}$ such that $M_n \rightarrow \infty$ and $\epsilon_n < \delta_{1/9}$, and $a = \frac{\epsilon_n}{3M_n}$, then there is a subset of $F_{S_n}(a)$, which forms an $\epsilon_n$-net of $\PP_{n,\alpha}$ and
\begin{equation}\label{equ:coverpp}
\displaystyle N(\epsilon_n, \PP_{n,\alpha}, d_{\infty}) \leq \left( 21^{2D} \right)^{3M_n/\epsilon_n} \left( \frac{18\sqrt{2D}\text{D}(M)}{\epsilon_n} \right)^{2D}.
\end{equation}
\end{lemma}

\begin{proof}
Given $f\in \PP_{n,\alpha}$, an approximation $\hat{f} \in F_{S_n}(a)$ is determined uniquely by specifying its boundary value $\hat{f} (ka)$ for $0 \leq k < 1/a$, which is given by
\[
\displaystyle \hat{f} (ka) = \text{arg min}_{x\in S_n} \dist_M (x, f(ka)).
\]
To show that $d_{\infty} (f, \hat{f}) \leq \epsilon_n$, We check the inequality $\dist_M (f(t), \hat{f}(t)) \leq \epsilon_n$ for all $t \in [0,1]$. Suppose $t \in [ka, ka+a]$. Since $\| f\|_{\alpha} \leq M_n$,
\begin{equation}\label{eq:lemma:function1:1}
\dist_M (f(ka), f(t)) \leq M_n a \leq \epsilon_n/3.
\end{equation}
Moreover, because $\hat{f}$ is a mapping to a geodesic on $[ka, ka+a]$ and the fact that $S_n$ is $\epsilon_n/3$-net of $M$,
\begin{equation}\label{eq:revision1}
\dist_M (f(ka), \hat{f}(ka)) \leq \epsilon_n/3
\end{equation}
and
\begin{align}
\label{eq:lemma:function1:2}
\dist_M (f(ka), \hat{f}(t)) &\leq \dist_M (f(ka), \hat{f}(ka+a)) \\ \nonumber
 &\leq \dist_M (f(ka), f(ka+a))+\dist_M (f(ka+a), \hat{f}(ka+a)) \\ \nonumber
 &\leq 2\epsilon_n/3.
\end{align}
It follows from~\eqref{eq:lemma:function1:1}, \eqref{eq:lemma:function1:2} and the triangle inequality that
\begin{equation}\label{eq:revision2}
\dist_M (f(t), \hat{f}(t)) \leq \epsilon_n.
\end{equation}
Define subset of $F_{S_n}(a)$:
\[
SF_{S_n}(a) = \{ f\in F_{S_n}(a) | \dist_M (f(ka), f(ka+a)) \leq \epsilon_n,\, \forall 0\leq k < 1/a \}.
\]
By the definitions of $\hat{f}$ and $S_n$ and~\eqref{eq:lemma:function1:1}, we conclude that $\hat{f} \in SF_{S_n}(a)$. Thus, $SF_{S_n}(a)$ is an $\epsilon_n$-net of $\PP_{n,\alpha}$.

By definition, $N(\epsilon_n, \PP_{n,\alpha}, d_{\infty}) \leq \# SF_{S_n}(a)$. It is thus sufficient to estimate $\# SF_{S_n}(a)$. Lemma~\ref{lemma:coverbound} and~\eqref{equ:cover} imply that
\[
\displaystyle \#S_n \leq \left( \frac{18\text{D}(M)}{\epsilon_n/\sqrt{2D}} \right)^{2D},
\]
which is the upper bound of the number of values that $\hat{f} (0)$ can take. Given the value of $\hat{f} (ka)$, there are $21^{2D}$ choices for $\hat{f} (ka+a)$ by Lemma~\ref{lemma:neighbor}. Thus, for $a = \epsilon_n/(3M_n)$, \eqref{equ:coverpp} is concluded as follows
\[
\# SF_{S_n}(a) \leq \left( 21^{2D}\right)^{3M_n/\epsilon_n} \left( \frac{18\sqrt{2D}\text{D}(M)}{\epsilon_n} \right)^{2D}.
\]
\end{proof}

\subsubsection{Covering Numbers of $\DD_{n,\alpha}$}
\label{sec:eps_net_D}

In this section, we prove the following lemma.
\begin{lemma}
If $M_n$ satisfies
\begin{equation}\label{equ:condition1}
\displaystyle M_n \leq \frac{n\epsilon_n^3}{6C_1D(\log(21)+1)},
\end{equation}
where $C_1$ was defined in Lemma~\ref{lemma:boundp}, then for $n$ sufficiently large $\DD_{n,\alpha}$ satisfies the inequality~2.2 of~\cite[Theorem~2.1]{ghosal2000}, that is,
\begin{equation}\label{equ:ghosal2d2}
\log N(\epsilon_n, \DD_{n,\alpha}, d_{q, \DD}) \leq n\epsilon_n^2.
\end{equation}
\end{lemma}

\begin{proof}

Recall that $\DD_{n,\alpha} = \Phi(\PP_{n,\alpha})$ and $d_{q, \DD} (p_{f_1}, p_{f_2}) \leq C_1 d_{\infty} (f_1, f_2)$ (see Lemma~\ref{lemma:boundp}). A consequence of this is that an $\epsilon_n$-net of $\DD_{n,\alpha}$ can be induced from an $\epsilon_n/C_1$-net of $\PP_{n,\alpha}$. Therefore,
\begin{equation}\label{eq:enetDDn}
\displaystyle N(\epsilon_n, \DD_{n,\alpha}, d_{q, \DD}) \leq N(\epsilon_n/C_1, \PP_{n,\alpha}, d_{\infty}) \leq \left( 21^{2D} \right)^{3C_1M_n/\epsilon_n} \left( \frac{18C_1\sqrt{2D}\text{D}(M)}{\epsilon_n} \right)^{2D}.
\end{equation}

To conclude~\eqref{equ:ghosal2d2}, it is enough to show that
\begin{equation}\label{eq:revision3}
\displaystyle \frac{3C_1M_n}{\epsilon_n} (2D\log(21)+2D) + 2D\log\left( \frac{18C_1\sqrt{2D}\text{D}(M)}{3C_1M_n} \right) \leq n\epsilon_n^2.
\end{equation}
We verify it for $n$ sufficiently large. Since $M_n \rightarrow \infty$, the second term of the LHS of~\eqref{eq:revision3} will be less than zero for large $n$. On the other hand, it follows from~\eqref{equ:condition1} that the first term of the LHS of~\eqref{eq:revision3} is less than or equal to $n\epsilon_n^2$.

\end{proof}

\subsection{Verification of Inequality~2.3 of~\cite{ghosal2000}}\label{sec:density2}

Recall that the prior $\Pi_n$, with support on $PGF(b_n)\subset \PP_{\alpha}$, is given by the discretized Brownian motion at times $b_n, 2b_n,\ldots, 1$. More specifically, we define the prior $\Pi_n$ on $PGF(b_n)$ by fixing the joint distribution of $f(kb_n)$ for $0\leq k < 1/b_n$, whose density is given by
\begin{equation}\label{equ:discretedensity}
\pi_n (f) = s(f(0))\prod_{k=0}^{1/b_n} p_{b_n}(f(kb_n), f(kb_n+b_n)).
\end{equation}
where $s$ is a fixed density function with support on $M$ for $f(0)$, and $p_{b_n} (x,y)$ is the transition probability from $x$ to $y$ of the Brownian motion at time $b_n$.

In this section, we show that if the sequence $b_n$ is properly chosen, then $\Pi_n$ satisfies the inequality~2.3 of~\cite[Theorem~2.1]{ghosal2000}, that is,
\begin{equation}\label{equ:ghosal2d3}
\Pi_n (\DD_{\alpha} \backslash \DD_{n,\alpha}) \leq \exp [-n \epsilon_n^2 (C+4)].
\end{equation}



We first establish Lemma~\ref{lemma:density2} below and then use it to conclude~\eqref{equ:ghosal2d3} in Lemma~\ref{lemma:function2} below (under a condition on $b_n$). We use the following set
\begin{equation}\label{equ:subset}
\begin{aligned}
X:=\{ f\in PGF(b_n) | \dist_M(f(k_1 b_n), f(k_2 b_n)) &\leq M_n (k_2 b_n - k_1 b_n)^{\alpha} /3,\\
&\, \forall 0\leq k_1 <k_2 < 1/b_n \}
\end{aligned}
\end{equation}

\begin{lemma}\label{lemma:density2}
The set $X$ is contained in $PGF(b_n)\cap \PP_{n,\alpha}$.
\end{lemma}

\begin{proof}
By definition of $\PP_{n,\alpha}$, it is enough to show that if $f\in X$, then
\[
\dist_M(f(t_1), f(t_2)) \leq M_n |t_2 - t_1|^{\alpha}, \quad \forall 0\leq t_1 < t_2 \leq 1.
\]
Suppose $t_1, t_2\in [k b_n, (k+1) b_n] $ for some $k$ without loss of generality. Since $f$ is geodesic on this interval and $f\in X$,
\begin{equation*}
\begin{aligned}
\displaystyle \dist_M(f(t_1), f(t_2)) &= \frac{| t_2 -t_1 |}{b_n}\dist_M(f(k b_n), f((k+1)b_n)) \\
 &\leq \frac{| t_2 -t_1 |^{\alpha}}{b_n^{\alpha}}\dist_M(f(k b_n), f((k+1)b_n)) \\
 & \leq \frac{M_n}{3} | t_2 -t_1 |^{\alpha}.
\end{aligned}
\end{equation*}
Now, let $t_1 \in [k_1 b_n, (k_1+1) b_n]$ and $t_2\in [k_2 b_n, (k_2+1) b_n]$ for $k_1 <k_2$. By the triangle inequality,
\begin{equation*}
\begin{aligned}
\displaystyle \dist_M(f(t_1), f(t_2)) &\leq \frac{M_n}{3} | k_1 b_n +b_n -t_1 |^{\alpha}+\frac{M_n}{3} | (k_2-k_1-1) b_n |^{\alpha}+ \frac{M_n}{3} | t_2 - k_1 b_n |^{\alpha} \\
 &\leq M_n | t_2 -t_1 |^{\alpha}.
\end{aligned}
\end{equation*}
This completes the proof.

\end{proof}

Next we consider the upper bound of the probability $\Pi_n (\PP_{\alpha}\backslash \PP_{n, \alpha})$. It uses a constant $C_2$ which is presented in Theorem~5.3.4 in~\cite[page 141]{sam}. It also introduces a constraint on $M_n$ and $\epsilon_n$ (see~\eqref{equ:condition2}).
\begin{lemma}\label{lemma:function2}
If $\frac{1}{2} \leq \alpha \leq 1$, $b_n= M_n^{-c}$ for a constant $c$ s.t. $0<c<1/\alpha$ and
\begin{equation}\label{equ:condition2}
C_2 \mathrm{Vol}(M)M_n^{c(2D+3)/2} \exp[-M_n^{2-(2\alpha-1)c}/18] \leq \exp[-n\epsilon_n^2 (C+4)],
\end{equation}
then~\eqref{equ:ghosal2d3} is satisfied.
\end{lemma}

\begin{proof}

We define
\[
X_{k_1, k_2} :=\{ f\in PGF(b_n) | \dist_M(f(k_1 b_n), f(k_2 b_n)) > M_n (k_2 b_n - k_1 b_n)^{\alpha} /3\}.
\]
When $\alpha\geq \frac{1}{2}$, Theorem~5.3.4 in~\cite[page 141]{sam} implies that for the constant $C_2>0$
\begin{equation}\label{equ:lemma:function2:1}
\Pi_n(X_{k_1, k_2}) \leq \frac{C_2}{b_n^{(2D-1)/2}} \exp[-b_n^{2\alpha} (M_n/3)^2/(2b_n)] \text{Vol}(M).
\end{equation}
Consequently,
\begin{align}\label{equ:lemma:function2:2}
\displaystyle \Pi_n(\PP_{\alpha}\backslash \PP_{n,\alpha}) &\leq \Pi_n(PGF(b_n)\backslash (PGF(b_n)\cap \PP_{n,\alpha})) \\ \nonumber
& \leq \Pi_n(PGF(b_n)\backslash X) \leq \sum_{0\leq i <j\leq \frac{1}{b_n}} \Pi_n(X_{k_i, k_j}) \\ \nonumber
 & \leq \frac{1}{b_n^2} \left( \frac{C_2}{b_n^{(2D-1)/2}} \exp[-b_n^{2\alpha} (M_n/3)^2/(2b_n)] \text{Vol}(M)\right) \\ \nonumber
& =\frac{C_2 \text{Vol}(M)}{b_n^{(2D+3)/2}} \exp[-b_n^{2\alpha-1} M_n^2/18].
\end{align}
The first inequality of~\eqref{equ:lemma:function2:2} follows from the fact that the support of $\Pi_n$ is $PGF(b_n)$. The second inequality of~\eqref{equ:lemma:function2:2} follows from Lemma~\ref{lemma:density2}. The third inequality follows from the definitions of $X$ and $X_{k_i, k_j}$. The fourth inequality of~\eqref{equ:lemma:function2:2} follows from~\eqref{equ:lemma:function2:1}. The proof concludes by plugging $b_n=M_n^{-c}$ in~\eqref{equ:lemma:function2:2} and the fact
$$
\Pi_n(\PP_{\alpha}\backslash \PP_{n,\alpha}) = \Pi_n(\DD_{\alpha}\backslash \DD_{n,\alpha}).
$$

\end{proof}

\subsection{Verification of Inequality~2.4 of~\cite{ghosal2000}}\label{sec:density3}

We recall that inequality~2.4 of~\cite[Theorem~2.1]{ghosal2000} states that
\begin{equation}\label{equ:ghosal2d4}
\Pi_n \left( P_0 \left(\log\frac{p_0}{p}\right) \leq \epsilon_n^2, P_0 \left(\log\frac{p_0}{p}\right)^2 \leq \epsilon_n^2 \right) \geq \exp [-n\epsilon_n^2 C].
\end{equation}
We first establish two technical lemmas (Lemmas~\ref{lemma:pupper} and~\ref{lemma:approx}) and then prove~\eqref{equ:ghosal2d4} in Lemma~\ref{lemma:density3}. The formulation of Lemma~\ref{lemma:pupper} requires the following notation. We recall that by choosing a density $p(t)$ on the predictor $t$, there is a map $\Phi: \PP \rightarrow \DD$. For simplicity, we use the following notation:
\[
p(t, x)=\Phi(f) =p_{\sigma^2}(f(t), x)p(t),\quad p_0(t, x)=\Phi(f_0)=p_{\sigma^2}(f_0(t), x)p(t),
\]
where $f$ is any continuous function and $f_0$ is the true function. Let $P_0$ be the probability with density $p_0(t,x)$ and $P_0 f$ denote $\int f dP_{f_0}$. Here the density $p(t)$ of the predictor $t$ is assumed to be positive on $[0,1]$, so that both $p(t,x)$ and $p_0(t,x)$ are positive (their exact forms are irrelevant). We consider first the upper bounds of $P_0 \left(\log\frac{p_0}{p}\right)$ and $P_0 \left(\log\frac{p_0}{p}\right)^2$.

\begin{lemma}\label{lemma:pupper}
There exists a constant $C_3>0$ such that
\begin{equation}\label{equ:lemma:pupper}
P_0 \left(\log\frac{p_0}{p}\right) \leq C_3 d_{\infty}(f_0, f), \quad
P_0 \left(\log\frac{p_0}{p}\right)^2 \leq C_3 d_{\infty}(f_0, f).
\end{equation}
\end{lemma}

\begin{proof}
Theorem~4.1.1 in~\cite[page 102]{sam} states that $p_{\sigma^2} (x,y)$ is strictly positive on $M\times M$. Since $M\times M$ is compact, there exists two constants $c_1,\, c_2>0$ such that
\begin{equation}\label{equ:lemma:pupper:2}
c_1 \leq p_{\sigma^2} (x,y) \leq c_2
\end{equation}
for all $(x,y)\in M\times M$. Moreover, for the same reason, $p_{\sigma^2} (x,y)$ is uniformly continuous. that is, there exists a constant $c_3>0$,
\begin{equation}\label{equ:lemma:pupper:3}
| p_{\sigma^2} (x_1,x) - p_{\sigma^2} (x_2,x) | \leq c_3 \dist_M(x_1, x_2) \quad \forall x_1,x_2,x\in M.
\end{equation}
Then, the inequality $\log(x) \leq x-1$, \eqref{equ:lemma:pupper:2} and~\eqref{equ:lemma:pupper:3} imply that
\begin{align}
\displaystyle P_0 \left(\log\frac{p_0}{p}\right) &= \iint \log \left(\frac{p_0}{p}\right) p_0 d\mu(x) dt \leq \iint \left( p_0-p \right) \frac{p_0}{p} d\mu(x)dt\\ \nonumber
 &\leq \iint \frac{c_2 c_3}{c_1} \dist_M(f(t), f_0(t)) p(t) d\mu(x)dt \leq \frac{c_2 c_3}{c_1}\text{Vol}(M) d_{\infty}(f_0, f).
\end{align}
Similarly,
\begin{equation*}
\begin{aligned}
\displaystyle P_0 \left(\log\frac{p_0}{p}\right)^2 &= \iint \left[ \log \left(\frac{p}{p_0}\right)\right]^2 p_0 d\mu(y) dt \\
 &= \iint_{p>p_0} \left[ \log \left(\frac{p}{p_0}\right)\right]^2 p_0 d\mu(y) dt + \iint_{p<p_0} \left[ \log \left(\frac{p_0}{p}\right)\right]^2 p_0 d\mu(y) dt \\
 &\leq \iint_{p>p_0} \left(\frac{p-p_0}{p_0}\right)^2 p_0 d\mu(y) dt + \iint_{p<p_0} \left(\frac{p_0-p}{p}\right)^2 p_0 d\mu(y) dt \\
 &\leq \frac{c_3^2}{c_1^2}d_{\infty} (f_0, f).
\end{aligned}
\end{equation*}
Consequently,~\eqref{equ:lemma:pupper} is satisfied with $\displaystyle C_3 = \max (\frac{c_2 c_3}{c_1}\text{Vol}(M), c_3^2/c_1^2)$.

\end{proof}

\begin{lemma}\label{lemma:approx}
Assume that $C_3$ is an arbitrarily chosen positive constant. If $f_0$ is a Lipschitz continuous function with the Lipschitz constant $L>0$ and $f\in PGF(b_n)$ such that $f(kb_n)$ is in the $r_n$-ball $B(f_0(kb_n), r_n)$ on $M$, where $\displaystyle r_n=\frac{\epsilon_n^2}{3C_3}-\frac{2L b_n}{3}$, then $d_{\infty} (f_0, f) \leq \epsilon_n^2/C_3$.
\end{lemma}

\begin{proof}
Since $f_0$ is Lipschitz,
\begin{equation}\label{equ:db}
\dist_M (f_0(kb_n), f_0(kb_n+t)) \leq L b_n \quad \forall 0\leq k <1/b_n, \, 0\leq t\leq b_n.
\end{equation}
Since $f$ is geodesic on each interval $[kb_n, kb_n+b_n]$,
\begin{equation}\label{equ:db:1}
\dist_M (f(kb_n), f(t)) \leq \dist_M (f(kb_n), f(kb_n+b_n)) \quad \text{for }\, t\in [kb_n, kb_n+b_n].
\end{equation}
By $\dist_M (f_0(kb_n), f(kb_n)) \leq r_n$, \eqref{equ:db}, \eqref{equ:db:1} and the triangle inequality,
\begin{equation}\label{equ:db:2}
\dist_M (f(kb_n), f(t)) \leq 2r_n + L b_n.
\end{equation}
Similarly,
\begin{equation}\label{equ:db:3}
\dist_M (f(kb_n), f_0(t)) \leq r_n+ L b_n.
\end{equation}
Inequalities~\eqref{equ:db:2} and~\eqref{equ:db:3} imply that
\begin{equation}\label{equ:db:4}
\dist_M (f_0(t), f(t)) \leq 3r_n + 2L b_n = \epsilon_n^2/C_3.
\end{equation}
The proof is concluded by the fact that~\eqref{equ:db:4} is true for every $t$.

\end{proof}

\begin{lemma}\label{lemma:density3}
If $f_0$ is a Lipschitz continuous function, then there exists a sufficiently large constant $C_0>0$ such that if $b_n = C_0 \epsilon_n^2$ and $n \epsilon_n^{4+\delta} \rightarrow \infty$ ($\delta>0$), then the sequence of priors $\Pi_n$ satisfies~\eqref{equ:ghosal2d4} for all $n>N_0$ ($N_0$ depends on $\delta$).
\end{lemma}

\begin{proof}

By Lemma~\ref{lemma:pupper}, it is enough to show that
\begin{equation}\label{equ:transition1}
\Pi_n (f : C_3 d_{\infty} (f_0, f) \leq \epsilon_n^2) \geq \exp [-n \epsilon_n^2 C],
\end{equation}
where $C_3$ is the constant in Lemma~\ref{lemma:pupper}. Let $f\in PGF(b_n)$ and $L$ be the Lipschitz constant of $f_0$. It follows from Lemma~\ref{lemma:approx} that if
\begin{equation}\label{equ:transition2}
\dist_M (f_0(kb_n), f(kb_n)) \leq r_n = \frac{\epsilon_n^2}{3C_3}-\frac{2L b_n}{3} \quad \forall 0\leq k <1/b_n,
\end{equation}
then $d_{\infty} (f_0, f) \leq \epsilon_n^2/C_3$.

Moreover, we note that~\eqref{equ:db}, \eqref{equ:transition2} and the triangle inequality imply that
\begin{equation}\label{equ:transition3}
\dist_M (f(kb_n), f(kb_n+b_n)) \leq L b_n + 2r_n.
\end{equation}
It follows from Theorem~5.3.4 in~\cite[page 141]{sam} and~\eqref{equ:transition3} that for a constant $C_4>0$,
\begin{equation*}
\begin{aligned}
p_{b_n} (f(kb_n), f(kb_n+b_n)) &\geq \frac{C_4}{b_n^{D/2}}\exp\left[-\frac{\dist_M(f(kb_n), f(kb_n+b_n))^2}{2b_n}\right] \\
 &\geq \frac{C_4}{b_n^{D/2}}\exp\left[-\frac{(L b_n + 2r_n)^2}{2b_n}\right].
\end{aligned}
\end{equation*}

Recall that the support of $\Pi_n$ is $PGF(b_n)$. Therefore,
\begin{equation}\label{equ:transition4}
\begin{aligned}
&\Pi_n (f : C_3 d_{\infty} (f_0, f) \leq \epsilon_n^2) \leq \\
&\displaystyle \frac{\mathrm{Vol}(B (f_0(0), r_n))}{\mathrm{Vol}(M)} \prod_{1\leq k< 1/b_n} \left[ \frac{C_4}{b_n^{D/2}}\exp\left[-\frac{(L b_n + 2r_n)^2}{2b_n}\right] \mathrm{Vol}(B (f_0(kb_n), r_n)) \right].
\end{aligned}
\end{equation}

Since
\[
\text{Vol}(B (f_0(kb_n), r_n)) \geq C_5 r_n^{D}\quad \text{for a constant } C_5>0,
\]
the RHS of~\eqref{equ:transition4} is at least
\begin{equation}\label{equ:transition5}
\displaystyle \frac{1}{\mathrm{Vol(M)}}\left(C_5 r_n^{D}\right)^{1/b_n+1} \frac{C_4^{1/b_n}}{b_n^{D/(2b_n)}}\exp\left[-\frac{(L b_n + 2r_n)^2}{2b_n^2}\right].
\end{equation}
Plugging the expression of $r_n$ in~\eqref{equ:transition2} and $C_0 b_n =\epsilon_n^2 $ for a constant $C_0>0$, the logarithm of~\eqref{equ:transition5} being greater or equal to $-n\epsilon_n^2 C$ is simplified as
\begin{align}\label{equ:transition6}
\displaystyle \frac{1}{b_n} & \left[ -\log \left(\frac{C_4 C_5^{1+b_n}}{\mathrm{Vol}(M)^{b_n}}\right) - D(1+b_n) \log \left( \frac{C_0}{3C_3}-\frac{2L}{3}\right) -(\frac{D}{2}+Db_n)\log(b_n) \right] \\ \nonumber
&+\frac{1}{2} \left( \frac{2C_0}{3C_3} - \frac{L}{3} \right)^2 \leq n\epsilon_n^2 C.
\end{align}
We fix a constant $C_0 >0$ large enough so that for all $b_n$,
$$
\displaystyle -\log \left(\frac{C_4 C_5^{1+b_n}}{\mathrm{Vol}(M)^{b_n}}\right) - D(1+b_n) \log \left( \frac{C_0}{3C_3}-\frac{2L}{3}\right) \leq 0.
$$
The constant $C_0$ exists since $b_n \rightarrow 0$. Moreover, we note that since the fourth term of~\eqref{equ:transition6} is a constant, to satisfy~\eqref{equ:transition6}, it is enough to show that
\begin{equation}\label{equ:transition7}
\displaystyle -\frac{1}{b_n}\left(\frac{D}{2}+Db_n \right)\log(b_n)  \leq \frac{C}{2} n\epsilon_n^2 .
\end{equation}
Substituting $C_0 b_n =\epsilon_n^2$ in~\eqref{equ:transition7} yields the inequality
\begin{equation}\label{equ:transition8}
\displaystyle \frac{C_0}{\epsilon_n^2}K\left(\frac{D}{2}+\frac{D}{C_0}\epsilon_n^2 \right)\log\left( \frac{C_0^{1/K}}{\epsilon_n^{2/K}} \right)  \leq \frac{C}{2} n\epsilon_n^2 .
\end{equation}
We note that by using $\log(x) \leq x$, it is enough to show that
\begin{equation}\label{equ:transition9}
\displaystyle K\left(\frac{D}{2}+\frac{D}{C_0}\epsilon_n^2 \right)C_0^{1+1/K} \leq \frac{C}{2} n\epsilon_n^{4+2/K}.
\end{equation}
If we pick any $K>0$ such that $\displaystyle \frac{2}{K}<\delta$, then the right-hand side of~\eqref{equ:transition9} approaches infinity while the left-hand side is bounded. This implies that there exists a constant $N_0>0$ such that for all $n>N_0$, \eqref{equ:transition9} is satisfied, which guarantees that~\eqref{equ:transition1} and thus the lemma are true.

\end{proof}

\subsection{Conclusion of Theorem~\ref{theorem:rate}}\label{sec:conclude:dbm}

Under the assumptions that $\frac{1}{2} \leq \alpha \leq 1$, $0 < c < \frac{1}{\alpha}$ and $f_0$ is Lipschitz, we showed, in previous sections, that if we pick $b_n,M_n,\epsilon_n$ such that
\begin{equation}
\label{eq:main_conditions}
b_n = M_n^{-c}, \ b_n =C_0 \epsilon_n^2, \ n\epsilon_n^{4+\delta}\rightarrow \infty
\text{ and~\eqref{equ:condition1} \& \eqref{equ:condition2} hold},
\end{equation}
then Theorem~\ref{theorem:rate} follows directly from~\cite[Theorem~2.1]{ghosal2000}. In this section, we conclude the proof by solving the inequalities for parameters and showing the optimal choice of the sequence $\epsilon_n$ (which determines the contraction rate).

The first two equalities of~\eqref{eq:main_conditions} imply that
\begin{equation}\label{equ:m}
M_n = C_0^{-1/c} \epsilon_n^{-2/c}.
\end{equation}
Plugging~\eqref{equ:m} into~\eqref{equ:condition1} and simplifying the expression yields
\begin{equation}\label{equ:epsilon1}
6C_0^{-1/c}C_1 D(\log(21)+1) \leq n\epsilon_n^{3+2/c}.
\end{equation}
Plugging~\eqref{equ:m} into~\eqref{equ:condition2} and taking the logarithm of both sides (with simplification) results in the inequality
\begin{align}\label{equ:epsilon2}
\displaystyle &\left(-\log\left( C_0^{-(2D+3)/2}C_2\text{Vol}(M) \right) +(2D+3)\log(\epsilon_n)\right)\epsilon_n^{4/c-4\alpha+2} \\ \nonumber
&+\frac{1}{18} C_0^{-2/c+2\alpha-1} \geq n\epsilon_n^{4/c-4\alpha+4} (C+4).
\end{align}
We note that the first term of~\eqref{equ:epsilon2} approaches zero when $4/c-4\alpha+2>0$. Therefore, to satisfy~\eqref{equ:epsilon2}, we only need that the second term, which is a constant, is no less than the right-hand side. That is,
\begin{equation}\label{equ:epsilon3}
\displaystyle \frac{1}{18} C_0^{-2/c+2\alpha-1}\geq n \epsilon_n^{4/c-4\alpha+4} (C+4).
\end{equation}
If we pick $\alpha$, $c$ and $\epsilon_n$ so that the right-hand side of~\eqref{equ:epsilon3} approaches zero, then~\eqref{equ:epsilon2} is satisfied for large $n$. It follows from~\eqref{equ:epsilon1},~\eqref{equ:epsilon3} and the fact that $n \epsilon_n^{4+\delta} \rightarrow \infty$ that the constants $\alpha$ and $c$ need to satisfy
\[
\displaystyle 3+\frac{2}{c} \leq 4+\delta <  \frac{4}{c}-4\alpha+4.
\]
One choice is $c=\frac{2}{1+\delta}$ and $\alpha=\frac{1}{2}$. Under this choice, the sequence $\epsilon_n=n^{-1/(4+3\delta/2)}$ satisfies~\eqref{equ:epsilon1} and~\eqref{equ:epsilon3}. Since $\delta>0$ can be arbitrarily small, the best achievable contraction rate is
\[
\displaystyle \epsilon_n = n^{-1/4+\epsilon}\quad \text{for any fixed $\epsilon >0$}.
\]

\section{Proof of Theorem~\ref{theorem:wcbm}}\label{sec:proof:cbm}

We first prove a technical lemma (Lemma~\ref{lemma:nearpoints}) which requires some definitions and then conclude the proof of Theorem~\ref{theorem:wcbm}. Let $Q_{x_0,...,x_k}^{T,...,kT}$ be the Brownian bridge probability measure on the path space $\displaystyle V^{(k)}=\left\{ f\in C([0, T], M) : f(0)=x_0, f(iT)=x_i,\, \forall 0\leq i\leq k\right\}$. In particular, we denote by $Q_{x,y}^T$ the Brownian bridge probability measure on the path space $\displaystyle V=\left\{ f\in C([0, T], M) : f(0)=x, f(T)=y\right\}$.

\begin{lemma}\label{lemma:nearpoints}
If $x,y\in M$ s.t. $\dist_M(x,y) < \epsilon_0/2$, then there exists $T_0>0$ such that
\[
Q_{x,y}^T (\dist_M(f, x) \geq \epsilon_0) < 1, \quad \forall\, T\leq T_0,
\]
where $\displaystyle \dist_M(f, x) = \max_{t\in [0,T]} \dist_M(f(t),x)$. In other words, the Brownian bridge assumes positive measure over the subset of paths $\displaystyle \left\{ f\in V : f([0,T])\subset B(x, \epsilon_0)\right\}$.

\end{lemma}

\begin{proof}
Equation~2.6 in~\cite{BBRM90} implies that there exists $T_0>0$ such that if $T\leq T_0$, then
\begin{equation}\label{equ:lemma:nearpoints1}
T\log(Q_{x,y}^T (\dist_M(f, x) \geq \epsilon_0)) \leq -\epsilon_0^2+4 \dist_M(x,y)^2.
\end{equation}
That is,
\begin{equation}\label{equ:lemma:nearpoints2}
Q_{x,y}^T (\dist_M(f, x) \geq \epsilon_0) \leq \exp[(-\epsilon_0^2+4 \dist_M(x,y)^2)/T] <1.
\end{equation}
The first inequality in~\eqref{equ:lemma:nearpoints2} follows from~\eqref{equ:lemma:nearpoints1} and the second inequality follows from the assumption that $\dist_M(x,y) < \epsilon_0/2$.

\end{proof}

We now conclude the proof of Theorem~\ref{theorem:wcbm}. Recall that the Kullback-Leibler (KL) divergence between $p_{f_0}$ and $p_f$ is defined as
\[
\displaystyle d_{KL}(p_{f_0}, p_f) =\int_{[0,1]\times M} p_{f_0}\log\left( \frac{p_{f_0}}{p_f} \right) dtd\mu(x).
\]
A corollary of Theorem~6.1 in~\cite{bayes65} implies that if $\Pi$ assumes positive mass on any Kullback-Leibler neighborhood of $p_{f_0}$, then the posterior distribution is weakly consistent. Thus, it is enough to show that
\[
\Pi(\{f: d_{KL} (p_{f_0}, p_f) \leq \epsilon\}) >0, \quad \forall \epsilon>0.
\]
We note that Lemma~\ref{lemma:pupper} shows that $d_{KL}(p_{f_0}, p_f)$ is upper bounded by $d_{\infty} (f_0, f)$. Therefore, we only need to prove that
\[
\Pi(\{f: d_{\infty} (f_0, f) \leq \epsilon\}) >0, \quad \forall \epsilon>0.
\]

Fix a positive number $\epsilon_1 <\epsilon$. We consider a regular (e.g., equidistant) grid of $[0,1]$ with spacing $T$. We assume the regular grid satisfies the following conditions:
\begin{enumerate}\label{cond:theorem:wcbm1}
  \item $B(x_i, \epsilon_1)\subset B(x, \epsilon),\, \forall x\in f_0 ([iT, (i+1)T])$ and $x_i=f_0(iT)$,
  \item $\dist_M(x_i, x_{i+1}) < \epsilon_1/4$.
\end{enumerate}
The Lipschitz assumption of $f_0$ guarantees the existence of $T$. Indeed, Condition (1) is guaranteed by the triangle inequality of the metric $\dist_M$ and the Lipschitz assumption and Condition (2) is guaranteed by picking a sufficiently small $T$.

Given a positive number $\delta <\epsilon_1/24$, the triangle inequality implies that
\begin{equation}\label{equ:theorem:wcbm2}
B(\hat{x}_i, 2\epsilon_1/3)\subset B(x_i, \epsilon_1) \,\text{ and }\, \dist_M(\hat{x}_i, \hat{x}_{i+1})<\epsilon_1/3, \quad \forall \hat{x}_i \in B(x_i, \delta).
\end{equation}
Applying Lemma~\ref{lemma:nearpoints} to $\hat{x}_i$ and $\hat{x}_{i+1}$ implies that
$Q_{\hat{x}_0,...,\hat{x}_{1/T}}^{T,...,1}$ assumes positive measure over the set of paths
\[
\displaystyle V_{\hat{\mb{x}}}=\left\{ f : f(0)=x_0, f(iT)=\hat{x}_i, f([iT, (i+1)T])\in B(\hat{x}_i, 2\epsilon_1/3),\, \, \forall i\in [0, 1/T]\right\}.
\]
If $f\in V_{\hat{\mb{x}}}$, then for any $t\in [iT, (i+1)T]$,
\begin{equation}\label{equ:theorem:wcbm3}
f(t)\in B(\hat{x}_i, 2\epsilon_1/3)\subset B(x_i, \epsilon_1)\subset B(f_0(t), \epsilon).
\end{equation}
The first inclusion in~\eqref{equ:theorem:wcbm3} follows from~\eqref{equ:theorem:wcbm2} and the second inclusion in~\eqref{equ:theorem:wcbm3} follows from condition (1) of the regular grid. By definition,~\eqref{equ:theorem:wcbm3} implies that $V_{\hat{\mb{x}}} \subset \{f: d_{\infty} (f_0, f) \leq \epsilon\}$. Therefore,
\[
\displaystyle \Pi(\{f: d_{\infty} (f_0, f) \leq \epsilon\}) \geq \int_{\hat{\mb{x}}\in \prod_{i=0}^{1/T} B(x_i, \delta)} Q_{\hat{x}_0,...,\hat{x}_{1/T}}^{T,...,1} (V_{\hat{\mb{x}}}) \Pi_n (d\hat{\mb{x}})>0,
\]
where $\Pi_n$ is the probability measure of the discretized Brownian motion with spacing $b_n=T$ and $\hat{\mb{x}}=(\hat{x}_1,\ldots, \hat{x}_{1/T})^T$.

\section{Extensions of The Regression Framework}\label{sec:extension}

In this section, we briefly discuss two extensions of the current framework, where Theorem~\ref{theorem:rate} and~\ref{theorem:wcbm} equally apply. In Section~\ref{sec:unknown_sigma}, we consider the case where the variance $\sigma^2$ is unknown. Section~\ref{sec:more_general_p} explains how to possibly relax the assumption that $p(t)$ has a positive lower bound.

%
%

\subsection{The Case of Unknown Variance $\sigma^2$}
\label{sec:unknown_sigma}
The mapping $\Phi$ of~\eqref{map:contden} assumes that $\sigma^2$ is a fixed and known parameter. If it is unknown, the prior on it can be chosen as the uniform distribution on the interval $[1/A, A]$ for some constant $A>0$ (or other distributions as long as it is bounded away from zero and infinity).

Under this prior of $\sigma^2$, the probability density of $(t,x)$ is given by
\[
\displaystyle p_f (t,x) = \int_{\sigma^2\in [1/A, A]} p_{\sigma^2}(f(t), x) p(t) d \sigma^2.
\]
Since $p_{\sigma^2}(x, y)$ and its partial derivatives (w.r.t.~$x$ and $y$) are uniformly continuous in the variable $\sigma^2$ over the interval $[1/A, A]$, it is easy to see that Lemmas~\ref{lemma:boundp} and~\ref{lemma:pupper} still hold for this type of probability densities. Therefore, the contraction rate for the case of unknown variance is the same as the case of fixed variance.

\subsection{More General $p(t)$}
\label{sec:more_general_p}

Throughout the paper, we assume that the distribution of the predictor $t$ has a smooth density $p(t)$ on $[0, 1]$ with strict lower and upper bounds $0< m_p \leq M_p$. This assumption is used in Lemma~\ref{lemma:boundp}. Since $p(t)$ is continuous, the upper bound $M_p$ always exists, but the lower bound can be restrictive. We can relax the lower bound on $p(t)$ as follows. Let $r>0$ and $S_r = \{t\in[0,1 ] | p(t)\geq r\}$. By following the same arguments in the proof, we note that the posterior distribution contracts at the same rate to the true function when considering the $L_q$ norm of functions restricted to $S_r$.

\section{Numerical Demonstrations}\label{sec:exp}

In this section, we demonstrate the proposed Bayesian scheme and compare it with a kernel method for the simple manifold $\Sb^1$. We also investigate the effect of changing various parameters for this special case.

One reason of using $\Sb^1$ is its simplicity of visualization. Indeed, $\Sb^1$ can be identified with the interval $[0, 2\pi]$ and this makes it easy to plot the $\Sb^1$-valued functions. The other reason is that $\Sb^1$, as a Lie group, has the addition operator on it. Thus, the kernel method in Euclidean spaces directly applies to this situation, with special awareness of the issue of averaging (more specifically, the average of the points $0$ and $2\pi$ on $\Sb^1$ is $0$, not $\pi$).

For the discretized and continuous BM Bayesian schemes, we obtain the maximum a posteriori (MAP) probability estimators by implementing a simulated annealing (SA) algorithm on the corresponding posterior distributions. The starting state (function) of SA is defined as follows: the value $f(t)$ at time $t$ is the mode of all observed values, whose observation times are in $[t-0.05, t+0.05]$. For the discretized BM Bayesian scheme, the sidelength parameter $b_n$ is fixed to be $1/40$. For the kernel method, we use the Matlab code~\cite{ksr_web} implemented according to the Nadaraya-Watson kernel regression with the optimal bandwidth suggested by Bowman and Azzalini~\cite{ksr}.

We remark that we use Brownian motion of various scales and not the standard one, $BM_t$, assumed in the proof. Nevertheless, the convergence result clearly holds for any scaled Brownian motion $BM_{ct}$, where $c>0$. In fact, $c$ is an additional hyperparameter (see Section~\ref{sec:hyperc}).

\subsection{Comparison with kernel regression}
In the first experiment, we compare three estimators, namely, the discretized BM MAP (DBM) estimator, the continuous BM MAP (CBM) estimator and the kernel regression estimator (KER). We fix the scaling hyperparameter $c=0.01$ for DBM and CBM and the optimal bandwidth for KER. We generate datasets of $30$ observations according to the pdf $p_{f_0(t)}(x)$ defined in~\eqref{equ:model}, where $\sigma^2=0.1$ and $f_0 : [0,1]\rightarrow \Sb^1$ defined by
$$
f_0(t) := (t+0.5)^2,\quad \text{for }t\in [0,1].
$$
Figure~\ref{fig:models} shows the original function and its different estimators according to DBM, CBM and KER. The $L_1$ errors between the estimated functions and the true function are also displayed. Among them, the CBM achieves the minimal $L_1$ error.

\begin{figure}[htb!]
\centering
\includegraphics[width=.8\textwidth]{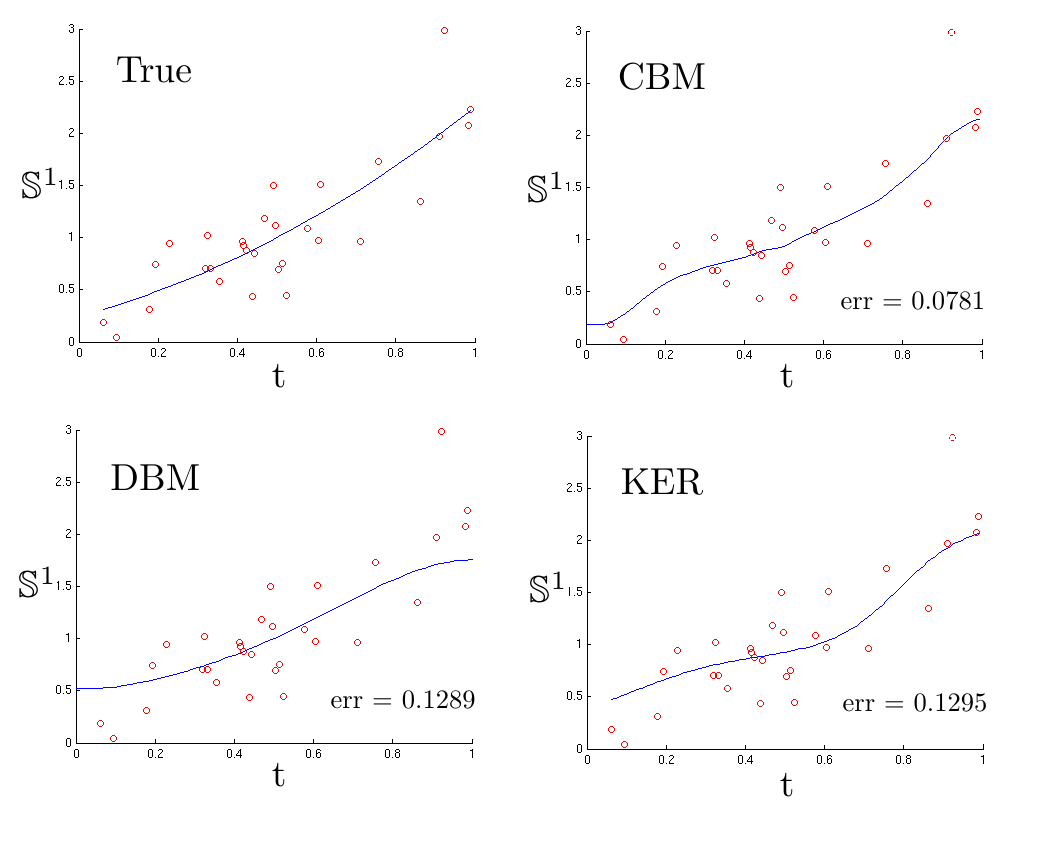}
\caption{Demonstration of the continuous and discretized BM Bayesian estimators with comparison to a kernel estimator. The data was generated according to the pdf $p_{f_0(t)}(x)$, where $f_0$ is demonstrated in the top left subfigure. The estimations obtained by CBM (continuous Brownian motion), DBM (discretized Brownian motion) and KER (kernel method) are shown in the rest of the subfigures together with their $L_1$ errors.}
\label{fig:models}
\end{figure}

\subsection{The hyperparameter $c$}\label{sec:hyperc}
The hyperparameter $c$ plays a similar role as the hyperparameter in the regularized regression. The second experiment shows how the hyperparameter $c$ (with values in $\{0.01, 0.1, 1, 10\}$) affects the estimation. We fix a dataset of 40 observations with noise variance $0.05$ from the same function as in the first experiment. Figures~\ref{fig:cbmc} and~\ref{fig:dbmc} demonstrate the MAP estimators obtained by CBM and DBM respectively.
In both figures, the estimators become smoother when $c$ decreases. Indeed, smaller $c$ means shorter time for the BM to travel. But smaller $c$ also introduces more bias in the estimators. This is more evident for DBM in Figure~\ref{fig:dbmc} while CBM seems less sensitive to small values of $c$ (see Figure~\ref{fig:cbmc}).

\begin{figure}[htb!]
\centering
\includegraphics[width=.8\textwidth]{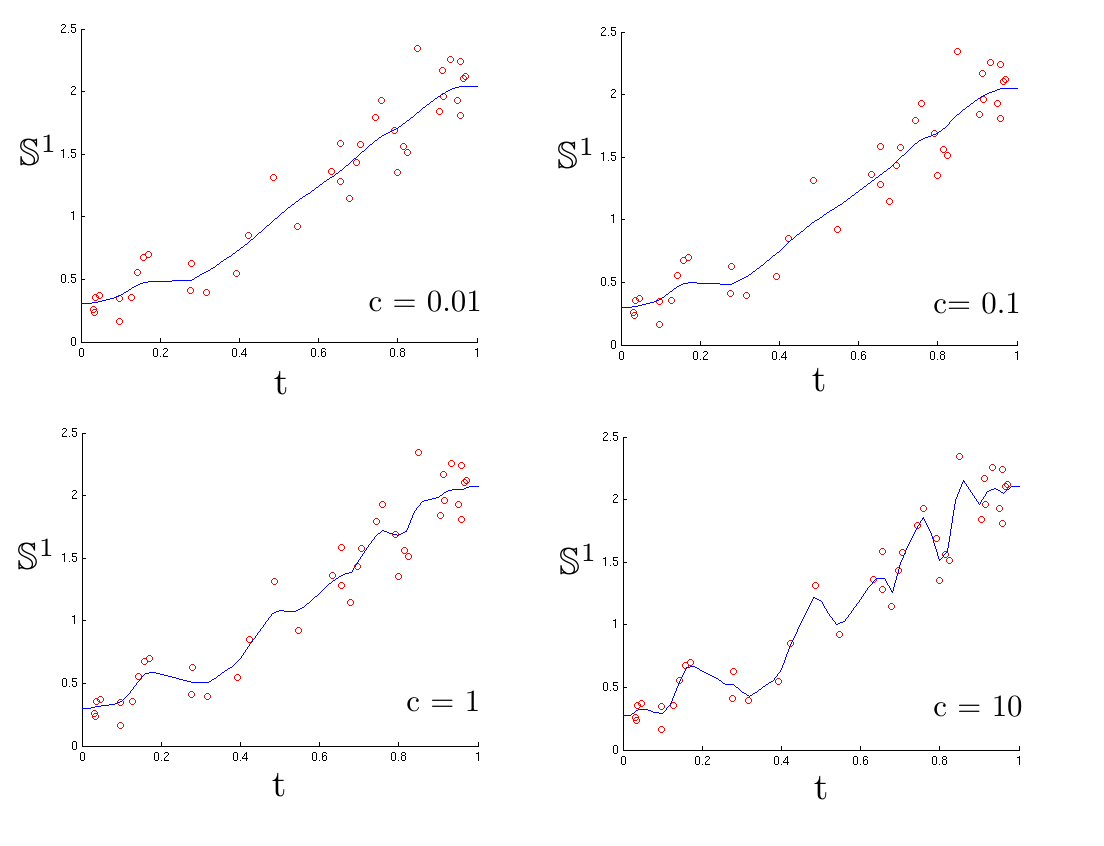}
\caption{The estimations obtained by CBM for different values of $c$ ($c=0.01,0.1,1,10$), where $c$ is the scaling parameter of the Brownian motion.}
\label{fig:cbmc}
\end{figure}

\begin{figure}[htb!]
\centering
\includegraphics[width=.8\textwidth]{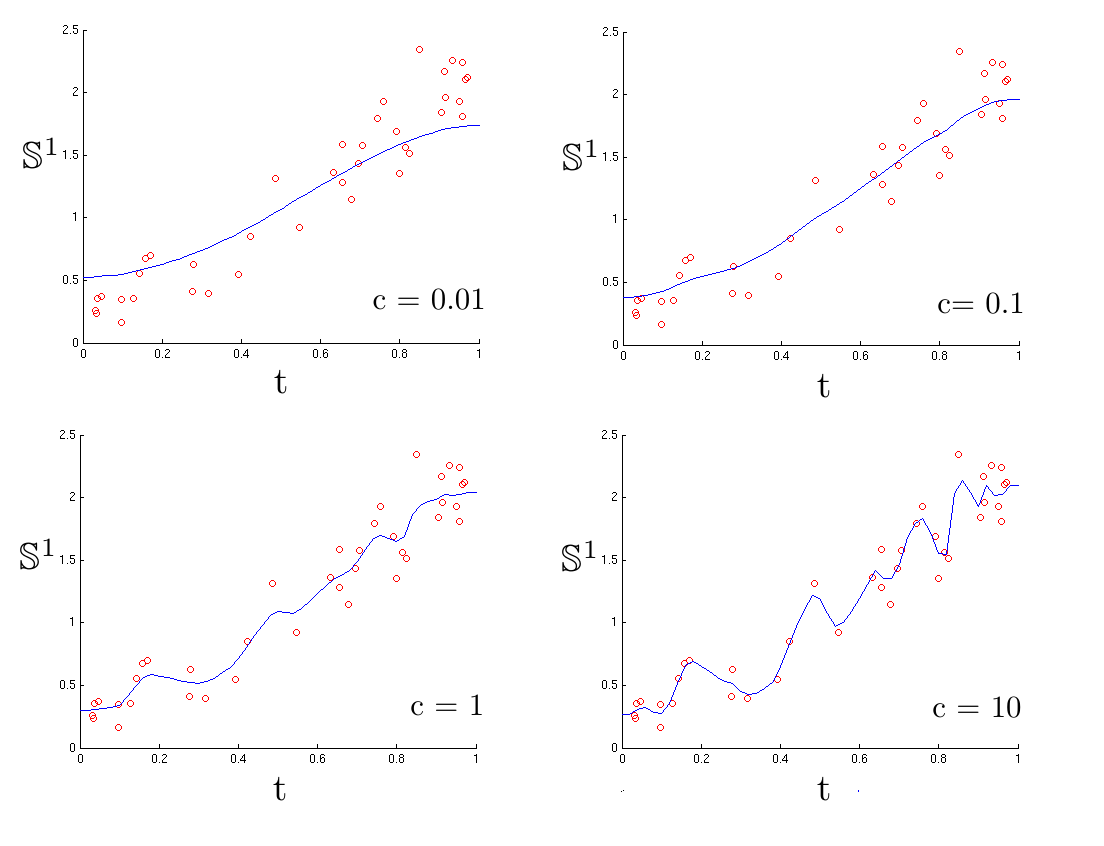}
\caption{The estimations obtained by DBM for different values of $c$ ($c=0.01,0.1,1,10$), where $c$ is the scaling parameter of the Brownian motion. Unlike CBM, an underestimation (i.e., sensitivity to bias) is observed when $c=0.01$.}
\label{fig:dbmc}
\end{figure}

\subsection{The sidelength parameter $b_n$}
For DBM we have another important parameter, $b_n$, which determines the number of pieces of a piecewise geodesic function. When $b_n = 1$, the piecewise geodesic function becomes geodesic. In this experiment, we show the change of $L_1$ error of the DBM estimator for different choices of $b_n$ ($1/b_n$ ranges from $1$ to $100$). The data set is generated from the same model as in the first experiment. Figure~\ref{fig:bn} shows that for geodesic functions or functions with large $b_n$, there is a large $L_1$ error due to large bias. As $b_n$ becomes smaller, there is a steady decrease of the $L_1$ error due to the decrease of bias.

 \begin{figure}[htb!]
\centering
\includegraphics[width=.8\textwidth]{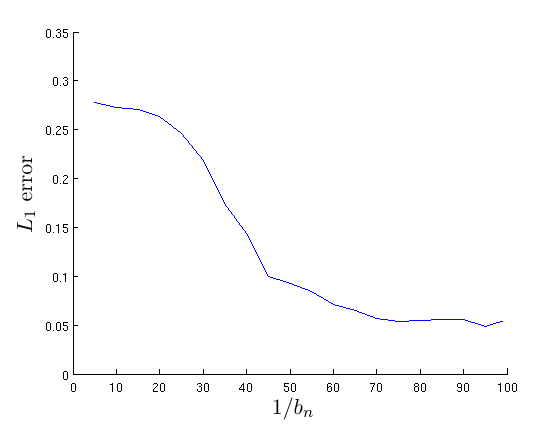}
\caption{$L_1$ error of DBM for different sidelengths $b_n$.}
\label{fig:bn}
\end{figure}


%
%
%

\section{Conclusion}

We established the consistency of the Bayesian estimator with a Brownian motion prior in the manifold regression setting. For the discretized Brownian motion, we even specified a contraction rate via a well-known general approach~\cite{ghosal2000, MR2418663}. We thus propose a new nonparametric Bayesian framework with solid statistical analysis beyond the existing kernel methods and Gaussian process priors. In fact, one of our motivations to this work is the incapability of applying a Gaussian process prior to manifold responses that lack linear structure.

We also list a few interesting questions for possible future study.

\subsection{Better Quantitative estimate of $C_0$ and $C_1$}

The constants $C_0$ and $C_1$ in Lemma~\ref{lemma:boundp} (comparing the distance of functions and the distance of distributions) are not specified due to our proof by contradiction. The specification of their dependencies on the underlying Riemannian geometry worth further investigation.

\subsection{$L_{\infty}$ Convergence}

We only proved $L_p$-convergence for the Brownian motion prior. It is interesting to investigate the $L_{\infty}$ convergence if it exists at all. If it does not exist, then it is interesting to know if a smoother prior (e.g., integrated BM) has $L_{\infty}$ convergence.

\subsection{A Better Contraction Rate?}
\label{sec:better_contraction}
For regression with real-valued predictors and responses, van Zanten~\cite{MR2791382} established posterior contraction rate of $n^{-1/4}$ for $n$ samples under the $L_q$-norm, where $1\leq q< \infty$. His analysis does not seem to extend to our setting. It is possible that even for the general case of manifold-valued regression the contraction rate is $n^{-1/4}$ and not just $n^{-1/4+\epsilon}$.
The particular method used here does not seem to obtain a better rate.

\end{document}